\pdfoutput=1
\documentclass[11pt]{amsart}
\usepackage[a4paper,margin=1.1in]{geometry}
\usepackage{amsmath,amssymb,amsthm,mathtools}
\usepackage{booktabs}
\usepackage{enumitem}
\usepackage{tikz}
\usetikzlibrary{arrows.meta,positioning,calc}
\usepackage[colorlinks=true,linkcolor=blue!50!black,citecolor=blue!50!black,urlcolor=blue!50!black]{hyperref}

\numberwithin{equation}{section}
\newtheorem{theorem}{Theorem}[section]
\newtheorem{proposition}[theorem]{Proposition}
\newtheorem{lemma}[theorem]{Lemma}
\newtheorem{corollary}[theorem]{Corollary}

\newtheorem{mainq}{Question}

\newtheorem*{claim*}{Main claim}
\theoremstyle{definition}
\newtheorem{definition}[theorem]{Definition}
\newtheorem{example}[theorem]{Example}
\theoremstyle{remark}
\newtheorem{remark}[theorem]{Remark}

\newcommand{\C}{\mathbb{C}}
\newcommand{\Z}{\mathbb{Z}}
\newcommand{\Q}{\mathbb{Q}}
\newcommand{\R}{\mathbb{R}}
\newcommand{\N}{\mathbb{N}}
\newcommand{\A}{\mathbb{A}}
\newcommand{\PP}{\mathbb{P}}
\newcommand{\OO}{\mathcal{O}}
\newcommand{\cA}{\mathcal{A}}
\newcommand{\cE}{\mathcal{E}}
\newcommand{\cF}{\mathcal{F}}
\newcommand{\cG}{\mathcal{G}}
\newcommand{\cI}{\mathcal{I}}
\newcommand{\cL}{\mathcal{L}}
\newcommand{\cQ}{\mathcal{Q}}
\newcommand{\cK}{\mathcal{K}}
\newcommand{\cV}{\mathcal{V}}
\newcommand{\m}{\mathfrak{m}}
\newcommand{\Rh}{\widehat{R}}
\DeclareMathOperator{\rk}{rk}
\DeclareMathOperator{\depth}{depth}
\DeclareMathOperator{\Spec}{Spec}
\DeclareMathOperator{\Proj}{Proj}
\DeclareMathOperator{\ch}{ch}
\DeclareMathOperator{\Tot}{Tot}
\DeclareMathOperator{\Bl}{Bl}
\DeclareMathOperator{\NS}{NS}
\DeclareMathOperator{\Sing}{Sing}
\DeclareMathOperator{\uc}{uc}
\DeclareMathOperator{\Tor}{Tor}
\newcommand{\iso}{\cong}

\expandafter\def\expandafter\UrlBreaks\expandafter{\UrlBreaks\do\-}
\begin{document}

\title[Counterexamples to the ESW problem and Hochster's conjecture]{The recent counterexamples to the Eisenbud--Schreyer--Weyman Ulrich problem (2003) and to Hochster's small Cohen--Macaulay conjecture (1973): a survey, with a generalization to arbitrary dimension}
\hypersetup{pdftitle={The recent counterexamples to the Eisenbud-Schreyer-Weyman Ulrich problem (2003) and to Hochster's small Cohen-Macaulay conjecture (1973): a survey, with a generalization to arbitrary dimension},pdfauthor={Cristian Anghel}}

\author{Cristian Anghel}
\address{Simion Stoilow Institute of Mathematics of the Romanian Academy, 21 Calea Grivi\c{t}ei, 010702 Bucharest, Romania}
\email{cristian.anghel@imar.ro}

\date{October 2026}


\begin{abstract}
In September 2026 four preprints and an AI-generated manuscript moved one numerical obstruction from projective geometry to local algebra. For a smooth complex projective surface $Y$ with an ample, globally generated divisor $H$ such that $H^2<\sigma(Y)$, the signature, one obtains in turn: no Ulrich bundles when $H$ is very ample (Anghel, arXiv:2609.02718; extended to every dimension by Lopez, arXiv:2609.06678), answering negatively the Eisenbud--Schreyer--Weyman question of 2003 for sheaves supported on the whole variety; no arithmetically Cohen--Macaulay bundles, i.e.\ section rings without graded maximal Cohen--Macaulay (MCM) modules in characteristic zero (Anghel, arXiv:2609.15589); and complete local domains without small Cohen--Macaulay modules, counterexamples to Hochster's 1973 conjecture, if $K_Y\equiv4H$ (Chen, arXiv:2609.24142, proposed) or $\sigma>2H^2$ (OpenAI, released 6 October). We recall the definitions and main statements, and add two things. First, through a finite projection to $\PP^2$ and an identity already present in the OpenAI manuscript, the three obstructions become one variance identity; for surfaces this explains the factor $2$ between the graded bound $\sigma\le H^2$ and the local $\sigma\le2H^2$. Second, we generalize the argument of the OpenAI manuscript to every dimension $m\ge2$, with complete proofs: if $X$ is a smooth complex projective variety of dimension $m$ and $H$ an ample, globally generated divisor, a small Cohen--Macaulay module over the completed section ring of $(X,H)$ forces $(K_X^2-2c_2(X))\cdot H^{m-2}\le(m+4)H^m$. Consequently the completed section rings of all Hesse pairs, including the ball quotient of the first preprint, have no small Cohen--Macaulay module, and products of Hesse surfaces with projective spaces give complete local domains with an isolated singularity, of every dimension $\ge3$, without small Cohen--Macaulay modules.
\end{abstract}

\maketitle

\setcounter{tocdepth}{1}
\tableofcontents
\section{Introduction}\label{sec:intro}

This survey tells a short story. In three weeks of September 2026 a numerical obstruction for vector bundles on polarized surfaces was pushed, one step at a time, from projective geometry to local commutative algebra. The central steps are recorded in three preprints and an AI-generated manuscript; a further preprint, [Lo], is a side branch towards higher dimension:
\begin{itemize}[leftmargin=2.2em]
	\item[{[An1]}] 2 September: smooth polarized surfaces with no Ulrich bundle \cite{An1};
	\item[{[Lo]}] 6 September (v3: 10 September): smooth polarized varieties of every dimension $\ge2$ with no Ulrich bundle \cite{Lo}; used in \cite[\S5]{An2}, not in the local arguments;
	\item[{[An2]}] 14 September: no arithmetically Cohen--Macaulay (aCM) bundles, hence section rings with no graded maximal Cohen--Macaulay (MCM) module \cite{An2};
	\item[{[Ch]}] 21 September: a proposed local cohomology obstruction for arbitrary modules over the completed section ring, under $K\equiv4H$ \cite{Ch};
	\item[{[OAI]}] 23 September (date of the manuscript): a complete local domain without a small Cohen--Macaulay module \cite{OAI}, published on 6 October in OpenAI's catalogue of mathematical manuscripts \cite{OAIrepo}.
\end{itemize}
Figure~\ref{fig:timeline} places them on a timeline, and Figure~\ref{fig:citations} records the citations among them (details in Section~\ref{sec:status}).

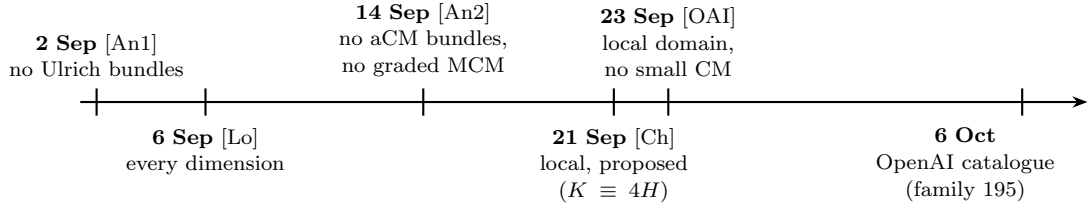
\begin{figure}[ht]
\centering
\begin{tikzpicture}[x=0.36cm,
	ev/.style={align=center,font=\scriptsize,text width=2.5cm,inner sep=1pt}]
	\draw[-{Stealth[length=2mm]},thick] (-0.6,0) -- (36.4,0);
	\foreach \d in {0,4,12,19,21,34} {\draw[thick] (\d,0.18) -- (\d,-0.18);}
	\node[ev,anchor=south] at (0,0.3) {\textbf{2 Sep} [An1]\\ no Ulrich bundles};
	\node[ev,anchor=north] at (4,-0.3) {\textbf{6 Sep} [Lo]\\ every dimension};
	\node[ev,anchor=south] at (12,0.3) {\textbf{14 Sep} [An2]\\ no aCM bundles, no graded MCM};
	\node[ev,anchor=north] at (19,-0.3) {\textbf{21 Sep} [Ch]\\ local, proposed ($K\equiv4H$)};
	\node[ev,anchor=south] at (21,0.3) {\textbf{23 Sep} [OAI]\\ local domain, no small CM};
	\node[ev,anchor=north east] at (35.5,-0.3) {\textbf{6 Oct}\\ OpenAI catalogue\\ (family 195)};
\end{tikzpicture}
\caption{The five texts (September--October 2026).}\label{fig:timeline}
\end{figure}

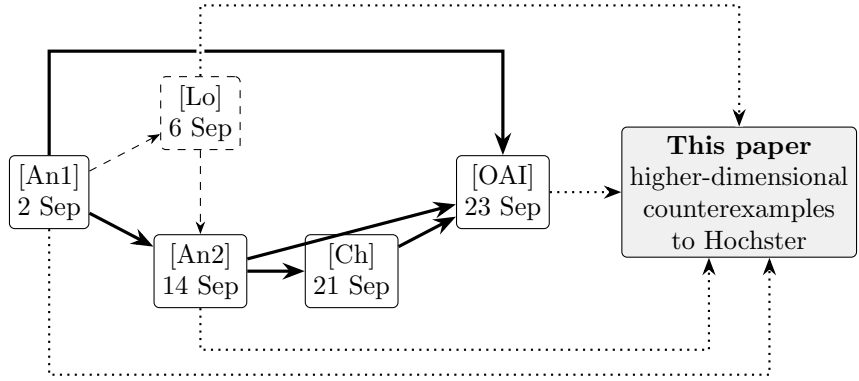
\begin{figure}[ht]
\centering
\begin{tikzpicture}[>=Stealth,x=1.25cm,y=1.1cm,
	paper/.style={draw,rounded corners=2pt,align=center,font=\small,inner sep=3pt}]
	\node[paper] (A1) at (0,0) {[An1]\\ 2 Sep};
	\node[paper,dashed] (L) at (1.6,0.95) {[Lo]\\ 6 Sep};
	\node[paper] (A2) at (1.6,-0.95) {[An2]\\ 14 Sep};
	\node[paper] (C) at (3.2,-0.95) {[Ch]\\ 21 Sep};
	\node[paper] (O) at (4.8,0) {[OAI]\\ 23 Sep};
	\draw[->,dashed] (A1) -- (L); \draw[->,dashed] (L) -- (A2);
	\draw[->,very thick] (A1) -- (A2); \draw[->,very thick] (A2) -- (C); \draw[->,very thick] (C) -- (O); \draw[->,very thick] (A2) -- (O);
	\draw[->,very thick] (A1.north) -- ++(0,1.25) -| (O.north);
	\node[paper,fill=black!6] (T) at (7.3,0) {\textbf{This paper}\\ higher-dimensional\\ counterexamples\\ to Hochster};
	\draw[->,dotted,thick] (O) -- (T);
	\draw[->,dotted,thick,preaction={draw,white,line width=3pt,-}] (L.north) -- ++(0,0.85) -| (T.north);
	\draw[->,dotted,thick] (A2.south) -- ++(0,-0.5) -| ([xshift=-4mm]T.south);
	\draw[->,dotted,thick] (A1.south) -- ++(0,-1.75) -| ([xshift=4mm]T.south);
\end{tikzpicture}
\caption{Citations among the five texts (an arrow points from a text to one that uses or cites it). The central chain is drawn thick; [Lo] (dashed) is a side branch towards higher dimension. Dotted arrows lead from the texts on which Section~\ref{sec:highdim} builds to this paper's higher-dimensional counterexamples to Hochster's conjecture: the argument of [OAI], the examples of [An1], the product construction of [An2], and the higher-dimensional Ulrich inequality of [Lo], which Section~\ref{sec:highdim} recovers.}\label{fig:citations}
\end{figure}

Three existence questions sit behind the story, and they form a tower (Figure~\ref{fig:ladder}). For a smooth projective surface $Y$ with a polarization $H$, let $R=R(Y,H)=\bigoplus_{t\ge0}H^0(Y,\OO_Y(tH))$ be the section ring and $\Rh$ its completion at the vertex. The three questions ask for the existence of
\begin{enumerate}[label=(R\arabic*),leftmargin=3em]
	\item an \emph{Ulrich bundle} on $(Y,H)$ --- the question of Eisenbud, Schreyer and Weyman \cite{ESW};
	\item an \emph{aCM bundle} on $(Y,H)$, equivalently a graded MCM module over $R$ --- the question Eisenbud and Schreyer asked after \cite{An1};
	\item a \emph{small Cohen--Macaulay module} over $\Rh$ --- Hochster's conjecture \cite{Ho73,Ho75,Ho17}.
\end{enumerate}
Existence propagates downwards: (R1)$\Rightarrow$(R2)$\Rightarrow$(R3) (Lemma~\ref{lem:ladder}). Non-existence propagates upwards, so each step of the story excludes a further level (\cite{Lo} extends the first level to higher dimension).

\begin{figure}[ht]
\centering
\begin{tikzpicture}[node distance=8mm,>=Stealth,
	rung/.style={draw,rounded corners=2pt,align=center,inner sep=4pt,font=\small,text width=7.3cm},
	tag/.style={font=\scriptsize,align=left}]
	\node[rung] (U) {(R1) Ulrich bundle on $(Y,H)$};
	\node[rung,below=of U] (A) {(R2) aCM bundle on $(Y,H)$ $\Longleftrightarrow$ graded MCM $R$-module};
	\node[rung,below=of A] (L) {(R3) small Cohen--Macaulay module over $\Rh$};
	\draw[->,thick] (U) -- (A) node[midway,left,font=\scriptsize]{Lemma \ref{lem:ladder}(a)};
	\draw[->,thick] (A) -- (L) node[midway,left,font=\scriptsize]{Lemma \ref{lem:ladder}(b)};
	\node[tag,right=3mm of U] {excluded by [An1], [Lo]};
	\node[tag,right=3mm of A] {excluded by [An2]};
	\node[tag,right=3mm of L] {excluded by [Ch] (proposed), [OAI]};
\end{tikzpicture}
\caption{The tower of existence questions.}\label{fig:ladder}
\end{figure}
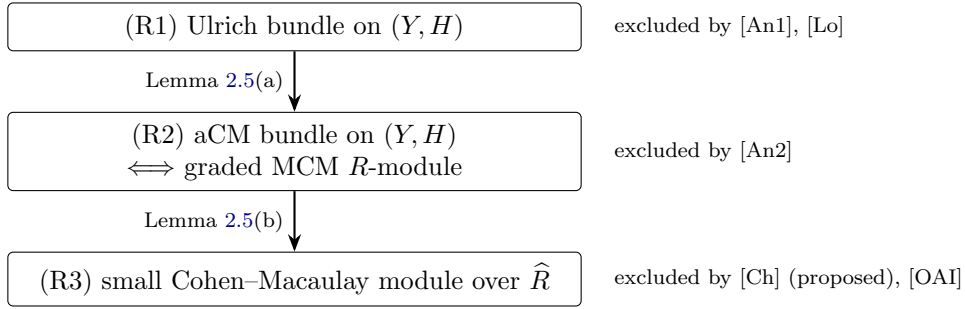

All obstructions compare the \emph{degree} and the \emph{signature},
\[
d=H^2,\qquad \sigma=\sigma(Y)=K_Y^2-8\chi(\OO_Y)=\tfrac13\bigl(K_Y^2-2c_2(Y)\bigr).
\]
Table~\ref{tab:summary} compares the numerical criteria, including Chen's proposed one.

\begin{table}[ht]
\centering
\small
\renewcommand{\arraystretch}{1.15}
\setlength{\tabcolsep}{3pt}
\begin{tabular}{@{}llll@{}}
	\toprule
	& object excluded & numerical criterion & hypotheses on $H$\\
	\midrule
	{[An1]} & Ulrich bundles & $H^2<\sigma$ & $H$ very ample\\
	{[Lo]} & Ulrich bundles, $\dim X=m$ & $(m+1)H^m<(K^2-2c_2)\cdot H^{m-2}$ & $H$ very ample\\
	{[An2]} & aCM bundles / graded MCM & $H^2<\sigma$ & $H$ ample, bpf\\
	{[Ch]} & MCM over $\Rh$ (proposed) & $H^2<\sigma$ & $H$ ample, bpf, $K\equiv4H$\\
	{[OAI]} & MCM over $\Rh$ & $2H^2<\sigma$ & $H$ ample, bpf\\
	Thm.~\ref{thm:local-m} & MCM over $\Rh$, $\dim X=m$ & $(m+4)H^m<(K^2-2c_2)\cdot H^{m-2}$ & $H$ ample, bpf\\
	\bottomrule
\end{tabular}
\caption{The five results, and the extension of \cite{OAI} proved in Section~\ref{sec:highdim} (bpf $=$ base-point-free).}\label{tab:summary}
\end{table}

Both local arguments take their examples from the complete-quadrangle family $(Y_n,3A_n-E_n)$ (Section~\ref{sec:kummer}): Hirzebruch's surfaces with the polarization $3A_n-E_n$ of \cite[\S7.2]{An1}, whose degree lies below the signature for $n\ge6$. \cite{OAI} takes the construction from \cite[\S7.2]{An1} and uses the prime $n=7$, and \cite{Ch} takes $n=6$ from \cite[Rem.~4.2, Cor.~4.3]{An2}, where its section ring is shown to be quasi-Gorenstein. For both values of $n$, \cite{An2} had already shown that the section ring has no graded MCM module. The Hesse pairs, the main examples of \cite{An1}, are used by neither local argument; see Corollary~\ref{cor:hesse-local}.

\subsection*{What this survey adds}
Besides a unified exposition, the survey makes two observations explicit and proves one extension.
\begin{enumerate}[label=(\alph*),leftmargin=2em]
	\item Through a finite map $f\colon Y\to\PP^2$ with $f^*\OO(1)=\OO(H)$, the three obstructions are one computation: a \emph{variance} of the push-forward, decomposed along a filtration with semistable factors (Proposition~\ref{prop:synthesis}). This identity keeps all terms in \cite[(10), (11)]{OAI} and contains \cite[Prop.~2.6]{An2}. What changes from level to level is what we know about the push-forward: trivial, split, or the restriction to an exceptional plane of a sheaf that is trivial off it. For surfaces this explains the factor $2$ between $\sigma\le d$ (graded) and $\sigma\le 2d$ (local, \cite{OAI}); Chen's argument fits the same computation, with a stronger lower bound for the variance (Remark~\ref{rem:chen-v}).
	\item By \cite[Thm.~1.2]{OAI} (proved in Section~\ref{sec:highdim}), the completed section rings of the Hesse pairs $(Y_n,4A_n-E_n)$, $n\ge3$, of the dual Hesse pairs $(Y_n,6A_n-E_n)$, $n\ge6$, and of the complete-quadrangle pairs $(Y_n,3A_n-E_n)$, $n\ge7$, have no small CM module (Corollary~\ref{cor:hesse-local}); this includes the ball quotient $Y_3$ of \cite{An1}.
	\item We generalize the argument of \cite{OAI} to every dimension $m\ge2$ (Theorem~\ref{thm:local-m}, with complete proof): for $X$ smooth projective of dimension $m$ and $H$ ample and globally generated, a small CM module over $\Rh(X,H)$ forces $(K_X^2-2c_2(X))\cdot H^{m-2}\le(m+4)H^m$, compared with the bound $(m+1)H^m$ for Ulrich and aCM bundles. With $X=Y_n\times\PP^{m-2}$, the products used in \cite[Cor.~5.3]{An2} for the graded statement, this gives complete normal local domains with an isolated singularity, of every dimension $\ge3$, without small CM modules (Corollaries~\ref{cor:products} and~\ref{cor:alldim}). Adjoining power series variables, the local analogue of a remark in \cite[\S5]{An2}, gives such domains with non-isolated singularities (Corollary~\ref{cor:powerseries}). In characteristic $p$ the completed section rings of smooth varieties always have small CM modules, by Hartshorne--Hochster--Peskine--Szpiro and Shimomoto--Tavanfar, so the characteristic hypothesis in these results is essential even locally (Remark~\ref{rem:charp}).
\end{enumerate}

\subsection*{Organization}
Sections~\ref{sec:questions}--\ref{sec:oai} recall definitions and statements, with at most brief indications of proof: the three questions and the tower (Section~\ref{sec:questions}), the Ulrich obstruction (Section~\ref{sec:ulrich}), the examples (Section~\ref{sec:kummer}), the aCM obstruction (Section~\ref{sec:acm}), the local setting (Section~\ref{sec:local}) and the two local arguments (Sections~\ref{sec:chen} and~\ref{sec:oai}). Sections~\ref{sec:synthesis} and~\ref{sec:highdim} contain the new material, with complete proofs: a single variance identity on $\PP^2$ from which the three obstructions are read off (Section~\ref{sec:synthesis}), and the extension of the local obstruction to every dimension (Section~\ref{sec:highdim}). Section~\ref{sec:status} records provenance and status, and Section~\ref{sec:questions-open} open questions.

\subsection*{Conventions}
We work over $\C$. Surfaces are smooth, connected and projective. For a coherent sheaf $\cF$ and a divisor $H$, $\cF(t)=\cF\otimes\OO(tH)$. For a torsion-free sheaf $\cG$ of rank $\rho$ on a polarized surface $(Y,H)$ we use the slope $\mu_H(\cG)=c_1(\cG)\cdot H/\rho$ and the normalized slope $\mu(\cG)=\mu_H(\cG)/H^2$; the discriminant is $\Delta(\cG)=2\rho c_2-(\rho-1)c_1^2$. Bogomolov's inequality says $\Delta\ge0$ for $\mu_H$-semistable torsion-free sheaves \cite[Thm.~3.4.1]{HL}.

\section{Three existence questions}\label{sec:questions}

This section states the three questions Q1--Q3 behind the levels (R1)--(R3) of the tower (Figure~\ref{fig:ladder}) and the implications between them.

\subsection{Ulrich bundles}
Let $X\subset\PP^N$ be a smooth projective variety of dimension $m$ and $H=\OO_X(1)$.

\begin{definition}
A vector bundle $\cE$ on $X$ is \emph{Ulrich} if $H^i\bigl(X,\cE(-p)\bigr)=0$ for all $i$ and all $1\le p\le m$.
\end{definition}

\begin{proposition}[{\cite[Prop.~2.1, Cor.~2.2]{ESW}}]\label{prop:ESW}
For a vector bundle $\cE$ on $X$ the following are equivalent:
\begin{enumerate}[label=(\roman*)]
	\item $\cE$ is Ulrich;
	\item for some (equivalently, every) finite linear projection $\pi\colon X\to\PP^m$, the sheaf $\pi_*\cE$ is trivial, $\pi_*\cE\iso\OO_{\PP^m}^{\oplus\rk\cE\cdot\deg X}$;
	\item the module $\Gamma_*(\cE)=\bigoplus_tH^0(X,\cE(t))$ is Cohen--Macaulay of dimension $m+1$ and has a linear minimal free resolution over $S=\C[x_0,\dots,x_N]$, generated in degree $0$.
\end{enumerate}
On a surface, $\cE$ is Ulrich if and only if it is aCM (Definition~\ref{def:acm}) and $\chi(\cE(t))=\frac12\rk\cE\cdot H^2\,(t+1)(t+2)$.
\end{proposition}

\begin{mainq}[Eisenbud--Schreyer--Weyman, 2003 {\cite[p.~543]{ESW}}]\label{q:ESW}
Does every embedded projective variety $X\subset\PP^N$ carry an Ulrich sheaf?
\end{mainq}

As in \cite[\S1]{An1}, following the restatement of Hanselka and Kummer \cite[Problem~4.3]{HK}, the sheaf is required to have support $X$; on smooth surfaces such sheaves are locally free \cite[Lemma~9.1]{An1}. Since $(X,\ell H)$ carries Ulrich bundles of rank two for $\ell\gg0$, for every smooth surface $X$ and ample $H$ \cite[Thm.~1.2]{CoHu}, Question~\ref{q:ESW} concerns a \emph{fixed} embedding. Following \cite{BES}, $\uc(X,H)$ denotes the smallest rank of an Ulrich bundle, and $\uc=\infty$ if there is none.

\subsection{aCM bundles and graded MCM modules}
Let $Y$ be a smooth surface and $H$ a divisor such that $|H|$ is base-point-free and $\varphi_H$ is finite; equivalently, $H$ is ample and globally generated.

\begin{definition}\label{def:acm}
A vector bundle $\cE$ on $Y$ is \emph{arithmetically Cohen--Macaulay} (aCM) with respect to $H$ if $H^1(Y,\cE(t))=0$ for every $t\in\Z$.
\end{definition}

Let $R=R(Y,H)=\bigoplus_{t\ge0}H^0\bigl(Y,\OO_Y(tH)\bigr)$ and $\m=R_+$. By \cite[Lemma~1.2]{An2}, $R$ is a finitely generated normal graded $\C$-domain of dimension $3$ with $\Proj R=Y$, and its punctured spectrum is smooth.

\begin{lemma}[{\cite[Lemma~3.2]{An2}}]\label{lem:acm-mcm}
The assignments $M\mapsto\widetilde M$ and $\cE\mapsto\Gamma_*(\cE)=\bigoplus_{t\in\Z}H^0(Y,\cE(t))$ are mutually inverse bijections between graded MCM $R$-modules and aCM bundles of positive rank on $(Y,H)$.
\end{lemma}

\begin{mainq}[Eisenbud--Schreyer, after \cite{An1}; see {\cite[Acknowledgements]{An2}}]\label{q:ES}
Do the polarized surfaces $(Y,H)$ constructed in \cite{An1} carry an aCM bundle; equivalently (Lemma~\ref{lem:acm-mcm}), do their section rings have a graded MCM module?
\end{mainq}
They also asked for the Boij--S\"oderberg cone of cohomology tables of aCM bundles on these surfaces. For $H^2<\sigma(Y)$, the cone spanned by the cohomology tables of vector bundles on $(Y,H)$ meets the tables with vanishing $h^1$-row only in $0$ \cite[Rem.~6.2]{An2}; see item~\ref{oq:bs} of Section~\ref{sec:questions-open}.

\subsection{Small Cohen--Macaulay modules}
Let $(R,\m)$ be a Noetherian local ring. A finitely generated $R$-module $M\ne0$ with $\depth_RM=\dim R$ is called a \emph{small} (or \emph{maximal}) \emph{Cohen--Macaulay module}.

\begin{mainq}[Hochster, 1973 {\cite{Ho73,Ho75}}]\label{conj:hochster}
Does every complete Noetherian local domain have a small Cohen--Macaulay module?
\end{mainq}
Hochster's \emph{small Cohen--Macaulay conjecture} predicts a positive answer. It holds in dimension $\le2$. In dimension three and characteristic $p>0$, every finitely generated $\N$-graded domain $R$ with $R_0$ a perfect field has a graded MCM module (Hartshorne, Hochster, Peskine--Szpiro \cite{Ho04,PS}), and the conjecture is known for further classes, such as three-dimensional complete weakly $F$-split rings \cite{Sch17}. In any characteristic there are existence criteria in terms of deficiency modules \cite[Thm.~3.2]{ShTa}, \cite{Xie1,Xie2}. In characteristic zero the conjecture was open from dimension three on, even for localizations of affine domains at maximal ideals \cite[\S1.3]{BHM}. The main motivation is Serre's positivity of intersection multiplicities \cite{Ho75}.

\subsection{The tower}

\begin{lemma}\label{lem:ladder}
Let $(Y,H)$ be as above and $\Rh$ the $\m$-adic completion of $R=R(Y,H)$.
\begin{enumerate}[label=(\alph*)]
	\item If $H$ is very ample, every Ulrich bundle on $(Y,H)$ is aCM.
	\item If $M$ is a graded MCM $R$-module, then $\widehat M=M\otimes_R\Rh$ is a small CM $\Rh$-module.
\end{enumerate}
\end{lemma}

Part (a) follows from Proposition~\ref{prop:ESW}(ii), and (b) from the faithful flatness of $R_\m\to\Rh$. Thus a complete local ring $\Rh(Y,H)$ without small CM modules has no graded MCM module over $R$, no aCM bundle and no Ulrich bundle.

\section{Degree, signature, and the Ulrich obstruction}\label{sec:ulrich}

For a smooth surface $Y$, Noether's formula $12\chi(\OO_Y)=K_Y^2+c_2(Y)$ gives
\[
\sigma(Y):=K_Y^2-8\chi(\OO_Y)=\frac{K_Y^2-2c_2(Y)}{3},
\]
the topological signature of $Y$. If $\sigma(Y)>0$ and $Y\not\iso\PP^2$, then $Y$ is of general type and, by the Bogomolov--Miyaoka--Yau inequality, $\sigma(Y)\le\chi(\OO_Y)$, with equality exactly for ball quotients; see \cite[Rem.~1.6]{An2}.

In a 2017 lecture \cite{Be17} Beauville observed that on a surface of Picard rank one Bogomolov's inequality for an Ulrich bundle gives $H^2\ge\sigma(S)$, and noted that it seems hard to obtain a counterexample to Question~\ref{q:ESW} this way. The Picard rank hypothesis is not needed.

\begin{proposition}[{\cite[Prop.~2.1]{An1}}]\label{prop:ulrich}
Let $(S,H)$ be a smooth surface with $H$ very ample. If $S$ carries an Ulrich bundle, then
\[
H^2\ \ge\ K_S^2-8\chi(\OO_S)=\sigma(S).
\]
\end{proposition}

\begin{proposition}[{\cite[Lemma~3.1]{Lo}}]\label{prop:lopez}
Let $X$ be smooth of dimension $m\ge2$ and $H$ very ample. If
\[
(m+1)H^m<\bigl(K_X^2-2c_2(X)\bigr)\cdot H^{m-2},
\]
then neither $(X,H)$ nor any of its successive smooth hyperplane sections (down to dimension $2$) carries an Ulrich bundle.
\end{proposition}

For $m=2$ this is Proposition~\ref{prop:ulrich}. The difference of the two sides, the excess $I(X,H)$ below, behaves simply under hypersurface sections.

\begin{lemma}[{\cite[Rem.~5.7]{An2}}]\label{lem:excess}
Let $X$ be smooth of dimension $m\ge2$ and $H$ ample, and set $I(X,H)=(K_X^2-2c_2(X))\cdot H^{m-2}-(m+1)H^m$. If $m\ge3$ and $X'\in|aH|$ is smooth and $H'=H|_{X'}$, then
\[
I(X',H')=a\bigl[I(X,H)-(a^2-1)H^m\bigr].
\]
\end{lemma}

\section{Hirzebruch--Kummer covers and the examples}\label{sec:kummer}

The examples of \cite{An1} are Hirzebruch's surfaces attached to line arrangements \cite{Hi,BHH}. Let $\cA=\{L_1,\dots,L_k\}$ be an arrangement of $k\ge4$ lines in $\PP^2$, not all concurrent, with defining forms $\ell_i$. Let $t_r$ be the number of points lying on exactly $r$ lines, and put
\[
f_0=\sum_{r\ge2}t_r,\qquad f_1=\sum_{r\ge2}rt_r,\qquad T=\sum_{r\ge3}t_r .
\]
The forms $\ell_i$ define a linear embedding $\PP^2\hookrightarrow\PP^{k-1}$ with image a plane $\Lambda$. For $n\ge2$ let $\Phi_n\colon\PP^{k-1}\to\PP^{k-1}$, $[z_i]\mapsto[z_i^n]$, and
\[
X_n=\Phi_n^{-1}(\Lambda),\qquad \pi_n=\Phi_n|_{X_n}\colon X_n\to\Lambda\iso\PP^2 .
\]

\begin{proposition}[{\cite[Props.~3.2, 3.4, 3.6, Cors.~3.5, 3.7]{An1}}]\label{prop:kummer}
\begin{enumerate}[label=(\alph*)]
	\item $X_n$ is a normal integral complete intersection, cut out by $k-3$ hypersurfaces of degree $n$; $\pi_n$ is finite of degree $n^{k-1}$, with $\pi_n^*\OO(1)=\OO_{X_n}(n)$.
	\item $X_n$ is smooth over points of $\PP^2$ lying on at most two lines. Over an $r$-fold point $p$ with $r\ge3$ it has $n^{k-r-1}$ singular points, each with completed local ring isomorphic to that of the vertex of the affine cone over a smooth complete-intersection curve $C_{n,r}\subset\PP^{r-1}$ cut out by $r-2$ diagonal hypersurfaces of degree $n$; $\deg C_{n,r}=n^{r-2}$ and $2g(C_{n,r})-2=n^{r-2}\bigl((r-2)n-r\bigr)$.
	\item Let $\Sigma_n=\Sing(X_n)_{\mathrm{red}}$ be the reduced singular locus. Blowing up $\Sigma_n$ gives a smooth surface $Y_n$; the exceptional curves are copies of $C_{n,r}$ with self-intersection $-n^{r-2}$, and $\cI_{\Sigma_n}\OO_{Y_n}=\OO_{Y_n}(-E_n)$ for the reduced exceptional divisor $E_n$.
	\item With $A_n$ the pull-back of $\OO_{X_n}(1)$ and $E_{n,r}$ the exceptional curves over $r$-fold points,
	\begin{gather*}
	A_n^2=n^{k-3},\qquad A_n\cdot E_n=0,\qquad E_n^2=-Tn^{k-3},\\
	K_{Y_n}=\bigl((k-3)n-k\bigr)A_n-\sum_{r\ge3}\bigl((r-2)(n-1)-1\bigr)E_{n,r}.
	\end{gather*}
	\item The signature is
	\[
	\sigma(Y_n)=\frac{n^{k-3}}{3}\bigl(Q_\cA n^2+R_\cA\bigr),\qquad
	Q_\cA=3-k+\sum_{r\ge2}(r-2)t_r,\quad R_\cA=k+3t_2-f_1-f_0 .
	\]
\end{enumerate}
\end{proposition}

\begin{figure}[ht]
\centering
\begin{tikzpicture}[scale=1.1]
	\clip (-0.45,-0.35) rectangle (2.85,2.3);
	\coordinate (p1) at (0,0); \coordinate (p2) at (2.4,0);
	\coordinate (p3) at (1.2,2.0); \coordinate (p4) at (1.2,0.72);
	\foreach \a/\b in {p1/p2,p1/p3,p1/p4,p2/p3,p2/p4,p3/p4}
	{\draw[shorten >=-3cm,shorten <=-3cm] (\a) -- (\b);}
	\foreach \p in {p1,p2,p3,p4} {\fill (\p) circle (1.6pt);}
	\foreach \q in {(1.2,0),(0.635,1.058),(1.765,1.059)} {\draw[fill=white] \q circle (1.6pt);}
\end{tikzpicture}
\caption{The complete quadrangle: $k=6$, four triple points (filled), three double points (open); $Q=1$, $R=-10$.}\label{fig:quadrangle}
\end{figure}
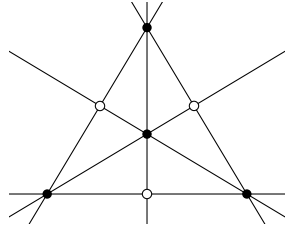

The three arrangements used below are the Hesse arrangement of the $12$ lines through pairs of the $9$ flexes of a smooth cubic ($t_4=9$, $t_2=12$, $Q=9$, $R=-33$), its dual ($k=9$, $t_3=12$, $Q=6$, $R=-39$), and the complete quadrangle of the six lines through four general points ($t_3=4$, $t_2=3$, $Q=1$, $R=-10$), see Figure~\ref{fig:quadrangle}. In all three, $\sigma(Y_n)$ grows like $n^{k-1}$, while $(aA_n-E_n)^2=(a^2-T)n^{k-3}$ grows like $n^{k-3}$.

\begin{lemma}[{\cite[Lemma~4.1]{An1}}]\label{lem:rees}
Let $X$ be projective with $\OO_X(1)$ very ample, $\cI\subset\OO_X$ a coherent ideal, $\rho\colon Y=\Bl_\cI X\to X$, $A=\rho^*\OO_X(1)$, and $\cI\OO_Y=\OO_Y(-E)$. Suppose $\cI(\nu)$ is generated by global sections for some $\nu\ge1$. Then the class $(\nu+1)A-E$ is very ample on $Y$.
\end{lemma}

\begin{proposition}[{\cite[Props.~5.1, 5.3]{An1}}]\label{prop:hesse-va}
For the Hesse arrangement and every $n\ge2$, the class $H_n=4A_n-E_n$ is very ample on $Y_n$. The class $3A_n-E_n$ is nef with $(3A_n-E_n)^2=0$, so $4$ is the smallest integral coefficient $a$ for which $aA_n-E_n$ can be ample.
\end{proposition}

The proof applies Lemma~\ref{lem:rees} with $\nu=3$: the four cubic monomials attached to the singular triangles of the Hesse pencil generate $\cI_{\Sigma_n}(3)$, for every $n$.

\begin{theorem}[{\cite[Thms.~1.1, 1.2, 5.6, \S5.1]{An1}}]\label{thm:an1}
For every $n\ge3$ the Hesse pair $(Y_n,H_n=4A_n-E_n)$ satisfies
\[
H_n^2=7n^9<(3n^2-11)n^9=\sigma(Y_n),
\]
so it carries no Ulrich bundle of any rank: $\uc(Y_n,H_n)=\infty$. The surface $Y_3$ is a compact ball quotient, with $K_{Y_3}=15A-3E$ ample, $K^2=144\cdot3^9=3c_2$, $\chi(\OO_{Y_3})=16\cdot3^9$. The surfaces $Y_n$ are pairwise non-isomorphic, since $K_{Y_n}^2=9n^9(n-1)(5n-7)$ is strictly increasing.
\end{theorem}

The same mechanism works for other arrangements admitting low-degree generators of the singular ideal \cite[Thm.~1.3]{An1}; Table~\ref{tab:families} lists the three families treated in \cite[\S\S5, 7, 8]{An1}.

\begin{table}[ht]
\centering
\small
\renewcommand{\arraystretch}{1.15}
\setlength{\tabcolsep}{4pt}
\begin{tabular}{@{}lccccc@{}}
	\toprule
	arrangement & $H_n$ & $H_n^2/n^{k-3}$ & $\sigma(Y_n)/n^{k-3}$ & no Ulrich for & ball quotient\\
	\midrule
	Hesse ($k=12$) & $4A-E$ & $7$ & $3n^2-11$ & $n\ge3$ & $n=3$\\
	dual Hesse ($k=9$) & $6A-E$ & $24$ & $2n^2-13$ & $n\ge5$ & $n=5$\\
	complete quadrangle ($k=6$) & $3A-E$ & $5$ & $(n^2-10)/3$ & $n\ge6$ & $n=5$\\
	\bottomrule
\end{tabular}
\caption{The three families of \cite{An1}; in each, $H_n^2/\sigma(Y_n)\to0$.}\label{tab:families}
\end{table}

Lopez combined Proposition~\ref{prop:lopez} with the Hesse pairs.

\begin{theorem}[{\cite[Thm.~1]{Lo}}]\label{thm:lopez}
Let $m\ge3$, let $(Y_n,H_n)$ be the Hesse pairs, $Z$ a smooth projective $(m-2)$-dimensional variety with very ample $M$, $X=Y_n\times Z$, $H=H_n\boxtimes M$, and $X'\in|aH|$ smooth with $a\ge1$. For $n\gg0$, the pairs $(X,H)$, $(X',H|_{X'})$ and all their successive hyperplane sections down to dimension $2$ carry no Ulrich bundle; $X'$ is not isomorphic to a nontrivial product.
\end{theorem}

\section{No aCM bundles: the graded theorem}\label{sec:acm}

Throughout this section $(Y,H)$ is a smooth surface with $|H|$ base-point-free and $\varphi_H$ finite; $d=H^2$, $\kappa=K_Y\cdot H$.

\begin{theorem}[{\cite[Thm.~1.1]{An2}}]\label{thm:acm}
If $(Y,H)$ carries a nonzero aCM vector bundle, then $H^2\ge\sigma(Y)$. Equivalently, if $H^2<\sigma(Y)$ there is no nonzero aCM bundle on $(Y,H)$, of any rank, semistable or not.
\end{theorem}

\begin{theorem}[{\cite[Thm.~1.3]{An2}}]\label{thm:gradedmcm}
If $H^2<\sigma(Y)$, then $R(Y,H)$ has no graded MCM module, $\depth_\m R=2$, and the vertex of $\Spec R$ is an isolated non-Cohen--Macaulay singularity. If $H$ is very ample, the homogeneous coordinate ring of $Y\subset\PP^N$ has no graded MCM module either.
\end{theorem}

The proof transports the problem to $\PP^2$, where aCM bundles split. Composing $\varphi_H$ with a general linear projection gives a finite flat morphism $\pi\colon Y\to\PP^2$ with $\pi^*\OO(1)=\OO_Y(H)$, $\deg\pi=d$, and $H^i(Y,\cG(t))=H^i(\PP^2,(\pi_*\cG)(t))$ for all $i$ and $t$.

\begin{lemma}[{\cite[Lemmas~2.1, 2.2]{An2}}]\label{lem:horrocks}
A vector bundle $\cE$ on $Y$ is aCM if and only if $\pi_*\cE\iso\bigoplus_{j=1}^N\OO_{\PP^2}(a_j)$, where $N=d\cdot\rk\cE$. For a coherent sheaf $\cG$ of positive rank on $Y$, with $h$ the line class of $\PP^2$,
\[
\mu_h(\pi_*\cG)=\frac{\mu_H(\cG)-\kappa/2}{d}-\frac32 .
\]
\end{lemma}

Let $\cE$ be an aCM bundle of rank $r$ with Harder--Narasimhan factors $\cQ_i$ of ranks $r_i$, weights $w_i=r_i/r$ and slopes $\mu_H(\cQ_1)>\dots>\mu_H(\cQ_s)$. Then $\pi_*\cE\iso\bigoplus_{j=1}^N\OO(a_j)$ has a filtration with successive quotients $\pi_*\cQ_i$, of slopes $b_i=(\mu_H(\cQ_i)-\kappa/2)/d-\frac32$, $b_1>\dots>b_s$. Put $\bar a=\frac1N\sum_ja_j=\sum_iw_ib_i$ and
\[
v_a=\frac1N\sum_j(a_j-\bar a)^2,\qquad v_b=\sum_iw_i(b_i-\bar a)^2 .
\]

\begin{lemma}[{\cite[Lemma~2.4]{An2}}]\label{lem:maj}
Let $V=\bigoplus_{j=1}^N\OO_{\PP^2}(a_j)$, and let $0=V_0\subset V_1\subset\dots\subset V_s=V$ be a filtration by subsheaves, with successive quotients of ranks $n_i>0$ and slopes $b_1\ge b_2\ge\dots\ge b_s$. Then for every $c\in\R$
\[
\sum_{i=1}^s\frac{n_i}{N}(b_i+c)^2\ \le\ \frac1N\sum_{j=1}^N(a_j+c)^2 .
\]
In particular $v_b\le v_a$.
\end{lemma}

For a torsion-free sheaf $\cF$ of rank $\rho$ on $Y$ write $\nu=c_1(\cF)/\rho\in\NS(Y)_\Q$, $\mu_H=\nu\cdot H$, $\delta=\Delta(\cF)/\rho^2$, $\beta=\nu-K_Y/2$, and $\alpha=\beta-\frac{\beta\cdot H}{d}H$, so that $\alpha\cdot H=0$ and $\alpha^2\le0$.

\begin{lemma}[{\cite[(2)]{An2}}]\label{lem:rr}
For all $t\in\Z$,
\[
\frac{\chi(\cF(t))}{\rho}=\frac d2\Bigl(t+\frac{\mu_H-\kappa/2}{d}\Bigr)^2-\frac{\sigma}{8}+\frac{\alpha^2-\delta}{2}.
\]
\end{lemma}

\begin{proposition}[Defect identity, {\cite[Prop.~2.6]{An2}}]\label{prop:defect}
Let $\cE$ be a nonzero aCM bundle with Harder--Narasimhan factors $\cQ_i$ and weights $w_i$, and let $v_a,v_b$ be as above. Then
\[
\frac{d-\sigma}{4}\;=\;d\,(v_a-v_b)\;+\;\sum_iw_i\bigl(\delta(\cQ_i)-\alpha(\cQ_i)^2\bigr).
\]
Every term on the right is $\ge0$: by Lemma~\ref{lem:maj}, by Bogomolov's inequality for the semistable factors, and by the Hodge index theorem. In particular $d\ge\sigma$, which proves Theorem~\ref{thm:acm}.
\end{proposition}

\begin{corollary}[{\cite[Cor.~1.4]{An2}}]\label{cor:hesse-graded}
For $n\ge3$, the section ring $R(Y_n,4A_n-E_n)$ of the Hesse pair is a three-dimensional normal graded $\C$-domain with an isolated non-Cohen--Macaulay singularity at the vertex, and it has no nonzero finitely generated graded MCM module. The same holds for the homogeneous coordinate ring of $Y_n\subset\PP^N$.
\end{corollary}

\begin{corollary}[{\cite[Cor.~4.3]{An2}}]\label{cor:qg}
For the complete quadrangle and $n\ge6$, let $H=3A_n-E_n$ on the quadrangle surface $Y_n$. Then $K_{Y_n}\sim(n-2)H$ (linear equivalence), so $R=R(Y_n,H)$ is quasi-Gorenstein with $\omega_R\iso R(n-2)$, and it has no graded MCM module. For $n=6$: $H^2=1080<\sigma=1872$, $K^2=17280$, $c_2=5832$, $\chi=1926$, $\omega_R\iso R(4)$.
\end{corollary}

\begin{proposition}[{\cite[Prop.~5.1]{An2}}]\label{prop:acm-highdim}
Let $X$ be smooth of dimension $m\ge2$ with $|H|$ base-point-free and $\varphi_H$ finite. If $X$ carries a nonzero aCM bundle $\cE$, that is, one with $H^i(X,\cE(t))=0$ for $1\le i\le m-1$ and all $t$, then $(m+1)H^m\ge(K_X^2-2c_2(X))\cdot H^{m-2}$.
\end{proposition}

With $X=Y_n\times\PP^{m-2}$ and $H=H_n\boxtimes\OO(1)$ the inequality fails as soon as $7(m-1)^2<3n^2-11$, and the Segre products $R(Y_n,H_n)\#\C[s_0,\dots,s_{m-2}]$ are normal graded domains of every dimension $m+1\ge3$ without graded MCM modules \cite[Cor.~5.3]{An2}.

\section{From graded to local}\label{sec:local}

\cite{An2} stops at level (R2) and leaves the local question open: Hochster's conjecture concerns the complete local ring $\Rh$, whose finitely generated modules need not be graded. This section fixes the framework of the two local arguments. Let $Y$ be a smooth surface and $H$ an ample, globally generated divisor; put $L=\OO_Y(H)$, $R=R(Y,H)$ and $\Rh$ its completion at the vertex. Three general sections $s_0,s_1,s_2$ of $L$ define a finite morphism $f\colon Y\to\PP^2$ with $f^*\OO(1)=L$ (the map $\pi$ of Section~\ref{sec:acm}); with $S=\C[x_0,x_1,x_2]\to R$, $x_i\mapsto s_i$, and $A=\C[[x_0,x_1,x_2]]$, the ring $\Rh=R\otimes_SA$ is finite over $A$.

\begin{proposition}[{\cite[Prop.~2.1]{Ch}, \cite[\S2]{OAI}}]\label{prop:model}
$\Rh$ is a three-dimensional complete Noetherian normal local domain with residue field $\C$. There is a proper morphism $p\colon W\to\Spec\Rh$ from a regular three-dimensional scheme $W$, obtained by base change from $\Tot(L^{-1})\to\Spec R$, such that
\begin{enumerate}[label=(\alph*)]
	\item the scheme-theoretic closed fibre is the zero section $Y\subset W$, an effective Cartier divisor, and $\OO_W(-Y)|_Y\iso L$, i.e.\ the conormal bundle of $Y$ is $\OO_Y(H)$;
	\item $p$ restricts to an isomorphism $U:=W\smallsetminus Y\iso\Spec\Rh\smallsetminus\{\m\}$;
	\item $\OO_W(-Y)$ is $p$-ample.
\end{enumerate}
\end{proposition}

In the version of \cite{OAI} one also keeps the base: $Z=\Bl_0\Spec A$ with exceptional plane $D\iso\PP^2$, and a finite morphism $g\colon W\to Z$ with $g^{-1}(D)=Y$, $g|_Y=f$, $\cI_Y=g^*\cI_D$ and $\cI_D|_D\iso\OO_D(1)$.

\begin{lemma}[{\cite[Lemmas~2.4, 2.5, \S6]{Ch}}]\label{lem:reflexive}
Let $M$ be a finite $\Rh$-module.
\begin{enumerate}[label=(\alph*)]
	\item If $M$ is small Cohen--Macaulay, then $M$ is torsion-free of full support and reflexive.
	\item If $M$ is reflexive, then $M$ is locally free on the punctured spectrum, $H^0_\m(M)=H^1_\m(M)=0$, and
	\[
	H^2_\m(M)\iso H^1\bigl(U,\widetilde M|_U\bigr).
	\]
\end{enumerate}
\end{lemma}

By (a) and (b), $\Rh$ has a small CM module if and only if there is a nonzero reflexive $M$ with $H^1(U,\widetilde M|_U)=0$. A reflexive $M$ is extended across $Y$ by the reflexive sheaf $\cF=(p^*\widetilde M)^{\vee\vee}$ on $W$, with $\cF|_U=\widetilde M|_U$ \cite[Lemma~2.4]{Ch}; its restriction $V=\cF|_Y$ is torsion-free \cite[Lemma~2.3]{Ch}. The extension is not unique, and both local arguments use this freedom.

\begin{lemma}[{\cite[Lemma~3.1]{Ch}} for $\cF$ reflexive; cf.\ the proof of {\cite[Lemma~3.1]{OAI}} for $\cF$ torsion-free]\label{lem:elem}
Let $\cF$ be a torsion-free coherent sheaf on $W$ whose restriction $V=\cF|_Y$ is torsion-free, and let $\varphi\colon V\to\cQ$ be a surjection onto a torsion-free sheaf with kernel $\cK$. Put $i\colon Y\hookrightarrow W$ and
\[
\cF'=\ker\bigl(\cF\to i_*V\xrightarrow{\,i_*\varphi\,}i_*\cQ\bigr).
\]
Then $\cF'$ is torsion-free (reflexive if $\cF$ is), $\cF'|_U=\cF|_U$, and its restriction to $Y$ fits into
\[
0\to\cQ\otimes\OO_Y(H)\to\cF'|_Y\to\cK\to0 .
\]
\end{lemma}

\begin{lemma}[Slope normalization]\label{lem:normalize}
Let $\cF$ be as in Lemma~\ref{lem:elem}.
\begin{enumerate}[label=(\alph*)]
	\item After finitely many modifications along the last Harder--Narasimhan quotient, all normalized Harder--Narasimhan slopes $\mu$ of $\cF|_Y$ lie in an interval of length at most $1$ \cite[Lemma~3.1]{OAI}.
	\item If $\cF$ is reflexive, then after a twist by a multiple of $Y$ and finitely many modifications, all normalized Harder--Narasimhan slopes of $\cF|_Y$ lie in the fixed window $[\frac32,\frac52)$, and $\cF$ remains reflexive \cite[Prop.~3.3]{Ch}.
\end{enumerate}
\end{lemma}

\section{Chen's local cohomology obstruction}\label{sec:chen}

Chen presents the following as a proposed argument, circulated for independent verification, and states its consequence for Hochster's conjecture conditionally.

\begin{claim*}[{\cite[Main claim~1.1]{Ch}}]
Let $Y$ be a smooth surface and $H$ an ample globally generated divisor with $K_Y\equiv4H$ and
\[
c(Y,H):=\tfrac{15}{8}H^2-\chi(\OO_Y)>0 .
\]
Then, for each nonzero finitely generated reflexive $\Rh$-module $M$, the local cohomology module $H^2_\m(M)$ is of finite length, with
\[
\dim_\C H^2_\m(M)\ \ge\ c(Y,H)\,\rk M .
\]
In particular $\Rh$ has no small Cohen--Macaulay module.
\end{claim*}

Under $K\equiv4H$ one has $8c(Y,H)=\sigma-d$, so $c>0$ is exactly the graded criterion $d<\sigma$ of Theorem~\ref{thm:acm}, now claimed for arbitrary modules. As \cite[Rem.~1.3]{Ch} notes, only the numerical equivalence is used; in Example~\ref{ex:chen} it holds linearly. The argument first uses Lemma~\ref{lem:normalize}(b) to put all normalized Harder--Narasimhan slopes of $V=\cF|_Y$ in the window $[\frac32,\frac52)$, centred at half the normalized slope $K_Y\cdot H/H^2=4$ of $\omega_Y$, and then combines the following three statements. Put $T_-(V)=\sum_{j<0}h^0(V(j))$ and $T_+(V)=\sum_{j>0}h^2(V(j))$.

\begin{proposition}[{\cite[Prop.~4.1]{Ch}}]\label{prop:chen-chi}
If $K_Y\equiv4H$ and $V$ is torsion-free of rank $r$ with all Harder--Narasimhan slopes in $[\frac32,\frac52)$, then $\chi(V)\le-c(Y,H)\,r$.
\end{proposition}

\begin{lemma}[{\cite[Lemma~4.2, Prop.~4.3]{Ch}}]\label{lem:tails}
If $K_Y\equiv4H$ and all Harder--Narasimhan slopes of $V$ lie in $[\frac32,\frac52)$, then $T_-(V)\le h^0(V)$ and $T_+(V)\le h^2(V)$.
\end{lemma}

\begin{proposition}[{\cite[Prop.~5.1, Cor.~5.4]{Ch}}]\label{prop:persist}
For a reflexive $\cF$ on $W$ with $V=\cF|_Y$ and $T_\pm(V)$ finite,
\[
\dim_\C H^1(U,\cF|_U)\ \ge\ h^1(V)-T_+(V)-T_-(V).
\]
\end{proposition}

With Lemma~\ref{lem:reflexive}(b) they give $\dim H^2_\m(M)\ge h^1(V)-h^2(V)-h^0(V)=-\chi(V)\ge c(Y,H)\,\rk M$.

\begin{example}[{\cite[Consequence~1.2, \S7]{Ch}}]\label{ex:chen}
Let $X_6\subset\PP^5$ be the complete intersection
\[
u_3^6-u_0^6+u_1^6=u_4^6-u_0^6+u_2^6=u_5^6-u_1^6+u_2^6=0,
\]
the exponent-$6$ Kummer cover of the complete quadrangle. Its blow-up $Y$ at the $144$ singular points carries $H=3A-E$, ample and globally generated, with $K_Y\sim4H$, $H^2=1080$, $\chi(\OO_Y)=1926$, so $c(Y,H)=99$. This is the ring of Corollary~\ref{cor:qg} with $n=6$, which \cite{Ch} takes from \cite[Rem.~4.2, Cor.~4.3]{An2}. Subject to the Main claim, $\Rh$ --- of depth $2$, quasi-Gorenstein, with an isolated singularity --- has no small CM module.
\end{example}

\section{OpenAI's complete local domain}\label{sec:oai}

The manuscript \cite{OAI}, dated 23 September 2026 and released on 6 October as family~195 of OpenAI's catalogue \cite{OAIrepo}, claims a complete proof, under the numerical hypothesis $\sigma(Y)>2H^2$, stronger than the graded one $\sigma(Y)>H^2$. According to the release notes \cite{OAIrepo}, it is the output of an OpenAI model and has no human author.

\begin{theorem}[{\cite[Thm.~1.1]{OAI}}]\label{thm:oai1}
There is a three-dimensional complete Noetherian normal local domain $R\supset\C$ with residue field $\C$ that has no nonzero finitely generated module of depth $3$.
\end{theorem}

\begin{theorem}[{\cite[Thm.~1.2]{OAI}}]\label{thm:oai2}
Let $Y$ be a smooth projective surface and $H$ an ample globally generated divisor. If $\Rh(Y,H)$ has a small Cohen--Macaulay module, then
\[
K_Y^2-2c_2(Y)\le6H^2,\qquad\text{equivalently}\qquad \sigma(Y)\le2H^2 .
\]
\end{theorem}

\begin{corollary}[{\cite[Cor.~7.1]{OAI}}]\label{cor:oai-affine}
For the example of Proposition~\ref{prop:oai-example} with $n=7$, the uncompleted local ring $R(Y,H)_{\m}$ --- a three-dimensional local domain essentially of finite type over $\C$ --- also has no small CM module.
\end{corollary}

The proof uses the finite morphisms $f\colon Y\to\PP^2$ and $g\colon W\to Z$ of Section~\ref{sec:local}. Let $M$ be a small CM module, choose an extension of $\widetilde M|_U$ normalized by Lemma~\ref{lem:normalize}(a), and let $P$ be its restriction to $Y$. Since $M$ is free over $A$ by the Auslander--Buchsbaum formula, $g_*$ of the extension is trivial off the exceptional plane $D$, and its restriction to $D$ is $f_*P$. For a coherent sheaf $F$ of rank $N>0$ on $\PP^2$ set
\[
v(F)=\frac{2\ch_2(F)}{N}-\Bigl(\frac{c_1(F)}{N}\Bigr)^2 ;
\]
for $F=\bigoplus\OO(a_j)$ this is the variance of $(a_j)$, and $v$ is invariant under twists.

\begin{lemma}[{\cite[Lemma~4.1]{OAI}}]\label{lem:variance}
Let $Z$ be an integral Noetherian scheme containing an effective Cartier divisor $D\iso\PP^2$ with $\cI_D|_D\iso\OO_D(1)$. If $\cL'$ is a torsion-free coherent sheaf on $Z$ with $\cL'|_{Z\smallsetminus D}\iso\OO^{N}$, then $v(\cL'|_D)\ge0$.
\end{lemma}

\begin{lemma}[{\cite[Lemma~5.1]{OAI}}]\label{lem:oai-ss}
Let $f\colon Y\to\PP^2$ be finite with $f^*\OO(1)=\OO(H)$, $d=H^2$. For every nonzero torsion-free $\mu_H$-semistable sheaf $\cG$ on $Y$,
\[
v(f_*\cG)\le\frac14-\frac{K_Y^2-2c_2(Y)}{12d}=\frac14-\frac{\sigma(Y)}{4d},
\qquad
\frac{c_1(f_*\cG)}{\rk f_*\cG}=\mu(\cG)-\frac{K_Y\cdot H}{2d}-\frac32 .
\]
\end{lemma}

Theorem~\ref{thm:oai2} follows by decomposing $v(f_*P)$ along the Harder--Narasimhan filtration of $P$; we carry this out in Section~\ref{sec:synthesis} (Proposition~\ref{prop:synthesis}), and in every dimension in Section~\ref{sec:highdim}.

\begin{proposition}[{\cite[Prop.~6.1]{OAI}}]\label{prop:oai-example}
For a prime $n\ge3$, let $\pi\colon Y\to B$ be the smooth $(\Z/n)^5$-cover of the blow-up $B$ of $\PP^2$ at four general points, obtained by adjoining $n$-th roots of the ratios of the six lines of the complete quadrangle. There is an ample, globally generated divisor $H$ on $Y$ with $nH\sim\pi^*(-K_B)$, and
\[
H^2=5n^3,\qquad K_Y^2-2c_2(Y)=n^3(n^2-10).
\]
\end{proposition}

For prime $n$, this is the complete-quadrangle surface of Table~\ref{tab:families}, as \cite[\S1]{OAI} states. Its polarization is $H=3J-E$ \cite[\S6.4, (16)]{OAI}, where the six sections of $\OO_Y(J)$ constructed there are the Kummer coordinates; hence $J=A_n$, $E=E_n$ and $H=3A_n-E_n$ \cite[\S7.2]{An1}. For $n=7$: $H^2=1715$ and $K^2-2c_2=13377>10290=6H^2$, so Theorem~\ref{thm:oai2} gives Theorem~\ref{thm:oai1}. Here $K\equiv5H$, so the Main claim of \cite{Ch} does not apply.

\section{Synthesis: one identity, three mechanisms}\label{sec:synthesis}

The three obstructions of Sections~\ref{sec:ulrich}, \ref{sec:acm} and~\ref{sec:oai} can be read as a single computation on $\PP^2$. We isolate this computation and prove it in full. The identity of Proposition~\ref{prop:synthesis} is implicit in \cite[(10), (11)]{OAI}, where it is used with inequalities, and its aCM case is the defect identity \cite[Prop.~2.6]{An2}; what we add is that all three obstructions are read off from it. Here $v$ is the variance of Section~\ref{sec:oai}, and for a torsion-free sheaf $\cG$ of rank $u$ on $Y$ we use $\nu$, $\mu_H=\nu\cdot H$, $\delta$, $\beta=\nu-K_Y/2$ and $\alpha=\beta-\frac{\beta\cdot H}{d}H$ as in Section~\ref{sec:acm}.

\begin{proposition}[cf.\ {\cite[(10), (11)]{OAI}}, {\cite[Prop.~2.6]{An2}}]\label{prop:synthesis}
Let $f\colon Y\to\PP^2$ be finite with $f^*\OO(1)=\OO_Y(H)$, $d=H^2$. Let $P$ be a torsion-free sheaf on $Y$ with a filtration whose successive quotients $G_i$ are torsion-free and $\mu_H$-semistable, with weights $w_i=\rk G_i/\rk P$ and push-forward slopes $b_i=c_1(f_*G_i)/\rk f_*G_i$, and let $v_b=\sum_iw_i(b_i-\bar b)^2$ with $\bar b=\sum_iw_ib_i$. Then
\[
v(f_*P)=\sum_iw_i\,v(f_*G_i)+v_b,\qquad
v(f_*G_i)=\frac14-\frac{\sigma}{4d}-\frac{\delta(G_i)-\alpha(G_i)^2}{d}\ \le\ \frac14-\frac{\sigma}{4d},
\]
with $\delta$, $\alpha$ as in Lemma~\ref{lem:rr}.
\end{proposition}

\begin{proof}
\emph{The decomposition.} The functor $f_*$ is exact, so $f_*P$ has a filtration with successive quotients $f_*G_i$, of ranks $N_i=w_iN$ with $N=\rk f_*P$. Rank, $c_1$ and $\ch_2$ are additive, so with $e_i=\ch_2(f_*G_i)/N_i$ one has $\bar b=c_1(f_*P)/N$ and
\[
v(f_*P)=\sum_iw_i\,2e_i-\bar b^2=\sum_iw_i\bigl(2e_i-b_i^2\bigr)+\Bigl(\sum_iw_ib_i^2-\bar b^2\Bigr)=\sum_iw_i\,v(f_*G_i)+v_b .
\]
\emph{The variance of one push-forward.} Let $\cG$ be torsion-free of rank $u$ on $Y$ and $F=f_*\cG$, of rank $ud$; put $b=c_1(F)/(ud)$ and $e=\ch_2(F)/(ud)$. Since $f$ is finite, $\chi(F(t))=\chi(\cG(tH))$ for all $t$. By Riemann--Roch on $\PP^2$, $\chi(F(t))/(ud)=\frac{t^2}2+\bigl(b+\frac32\bigr)t+1+\frac32b+e$. By Riemann--Roch on $Y$,
\[
\frac{\chi(\cG(tH))}{ud}=\frac{t^2}2+\frac{\mu_H-\kappa/2}{d}\,t+\frac1d\Bigl(\chi(\OO_Y)+\frac{\nu\cdot(\nu-K_Y)}2-\frac\delta2\Bigr),\qquad\kappa=K_Y\cdot H .
\]
The linear terms give $b+\frac32=(\mu_H-\kappa/2)/d=\beta\cdot H/d$. The constant terms give $2e=\frac2d\bigl(\chi(\OO_Y)+\frac12\nu\cdot(\nu-K_Y)-\frac\delta2\bigr)-2-3b$, hence
\[
v(F)=2e-b^2=\frac{2\chi(\OO_Y)+\nu\cdot(\nu-K_Y)-\delta}{d}-\Bigl(b+\frac32\Bigr)^2+\frac14 .
\]
Now $\nu\cdot(\nu-K_Y)=\beta^2-\frac14K_Y^2$ and $\beta^2=\alpha^2+(\beta\cdot H)^2/d$, so the terms $(\beta\cdot H)^2/d^2$ cancel and
\[
v(F)=\frac14+\frac{8\chi(\OO_Y)-K_Y^2}{4d}+\frac{\alpha^2-\delta}{d}=\frac14-\frac{\sigma}{4d}-\frac{\delta(\cG)-\alpha(\cG)^2}{d}.
\]
For $\cG=G_i$ semistable, $\delta(G_i)\ge0$ by Bogomolov's inequality and $\alpha(G_i)^2\le0$ by the Hodge index theorem, since $\alpha\cdot H=0$.
\end{proof}

The linear terms also show that $b_i-\mu_H(G_i)/d$ does not depend on $i$: the push-forward slopes are the slopes of the $G_i$, shifted and rescaled by the same amounts (Lemmas~\ref{lem:horrocks} and~\ref{lem:oai-ss}).

Hence $v(f_*P)\le\frac14-\frac{\sigma}{4d}+v_b$ always. Each result combines this upper bound with a lower bound for $v(f_*P)$ and an upper bound for $v_b$; Table~\ref{tab:synthesis} lists them.

\begin{table}[ht]
\centering
\small
\renewcommand{\arraystretch}{1.25}
\begin{tabular}{@{}llll@{}}
	\toprule
	& what is known about $f_*P$ & control of the slopes & conclusion\\
	\midrule
	Ulrich (Prop.~\ref{prop:ulrich}) & $f_*\cE\iso\OO^N$, so $v=v_a=0$ & majorization: $v_b\le v_a=0$ & $\sigma\le d$\\
	aCM (Thm.~\ref{thm:acm}) & $f_*\cE\iso\bigoplus\OO(a_j)$, so $v=v_a$ & majorization: $v_b\le v_a$ & $\sigma\le d$\\
	local (Thm.~\ref{thm:oai2}) & restriction of a sheaf $\iso\OO^N$ off $D$: & window of width $1$: & $\sigma\le2d$\\
	& $v\ge0$ (Lemma~\ref{lem:variance}) & $v_b\le\frac14$ & \\
	local, [Ch] & $\chi(f_*P)\ge0$ (claimed, $K\equiv4H$): & centred window: & $\sigma\le d$\\
	& $v\ge\frac14-(\bar b+\frac32)^2$ (Rem.~\ref{rem:chen-v}) & $\sum_iw_i(b_i+\frac32)^2\le\frac14$ & \\
	\bottomrule
\end{tabular}
\caption{One identity, three mechanisms.}\label{tab:synthesis}
\end{table}

\begin{enumerate}[label=(\roman*),leftmargin=2em]
	\item \emph{Ulrich.} For a linear projection, $f_*\cE\iso\OO^N$ (Proposition~\ref{prop:ESW}), so $v(f_*\cE)=v_a=0$. Lemma~\ref{lem:maj}, applied to the push-forward of the Harder--Narasimhan filtration, gives $v_b\le v_a=0$, so all $b_i$ coincide; since they are strictly decreasing, $\cE$ is $\mu_H$-semistable, which recovers \cite[Thm.~2.9]{CH}. Hence $0\le\frac14-\frac\sigma{4d}$. This is Proposition~\ref{prop:ulrich}; (i) is the case $v_a=0$ of (ii).
	\item \emph{aCM.} $f_*\cE$ splits, $v(f_*\cE)=v_a$, and Lemma~\ref{lem:maj} applied to the push-forward of the Harder--Narasimhan filtration gives $v_b\le v_a$. Hence $0\le v_a-v_b\le\frac14-\frac\sigma{4d}$. Keeping the equality of Proposition~\ref{prop:synthesis}, this is exactly the defect identity of Proposition~\ref{prop:defect}.
	\item \emph{Local.} Let $M$ be a small CM module over $\Rh(Y,H)$ and $P$ the restriction to the exceptional surface of an extension normalized by Lemma~\ref{lem:normalize}(a). We do not assume that $f_*P$ splits; but $f_*P$ is the restriction to the exceptional plane of a sheaf that is trivial off it, so $v(f_*P)\ge0$ by Lemma~\ref{lem:variance}. The normalized slopes $\mu_H(G_i)/d$ of the Harder--Narasimhan factors lie in an interval of length $\le1$, hence so do the $b_i$, and a weighted variance of points in such an interval is at most $\frac14$: with $\gamma$ the midpoint, $v_b\le\sum_iw_i(b_i-\gamma)^2\le\frac14$. Hence $0\le\frac12-\frac\sigma{4d}$, i.e.\ $\sigma\le2d$. Given Lemmas~\ref{lem:normalize}(a) and~\ref{lem:variance} and the identification of $f_*P$ in Section~\ref{sec:oai}, this proves Theorem~\ref{thm:oai2}; Section~\ref{sec:highdim} proves these ingredients as well, in arbitrary dimension.
\end{enumerate}

For surfaces, the factor $2$ between $\sigma\le d$ and $\sigma\le2d$ comes from replacing majorization, $v_b\le v_a$, available when the push-forward splits, by the bound $v_b\le\frac14$ for the variance of slopes in a window of width one. In dimension $m$ the two bounds become $(m+1)H^m$ and $(m+4)H^m$ (Section~\ref{ssec:identity}). Chen's argument (Section~\ref{sec:chen}) fits the same scheme: measuring $\chi$ instead of $v$ gives a stronger lower bound for $v(f_*P)$, at the price of a window centred at half the normalized slope of $\omega_Y$, so that Serre duality, $s\mapsto\lambda-s$ with $\lambda=K_Y\cdot H/d$, maps it into the closed window $[\frac\lambda2-\frac12,\frac\lambda2+\frac12]$ ($s\mapsto4-s$ and $[\frac32,\frac52]$ when $K\equiv4H$). In the subcanonical case $K\equiv4H$ it would recover the graded bound $\sigma\le d$ locally.

\begin{remark}[Chen's mechanism as a lower bound for $v$]\label{rem:chen-v}
For a coherent sheaf $F$ of rank $N>0$ on $\PP^2$ with $b=c_1(F)/N$, Riemann--Roch gives $\chi(F)/N=1+\frac32b+\frac12\bigl(v(F)+b^2\bigr)$, that is,
\[
\frac{\chi(F)}{N}=\frac12\Bigl(b+\frac32\Bigr)^2+\frac{v(F)}2-\frac18 .
\]
Let $P$ be a torsion-free sheaf on $Y$ with Harder--Narasimhan factors $G_i$, weights $w_i$, push-forward slopes $b_i$ with mean $\bar b$, and put $\lambda=K_Y\cdot H/d$. Since $\sum_iw_i(b_i+\frac32)^2=v_b+(\bar b+\frac32)^2$ and $b_i+\frac32=\mu(G_i)-\frac\lambda2$ (Lemma~\ref{lem:oai-ss}), Proposition~\ref{prop:synthesis} and $\chi(P)=\chi(f_*P)$ give
\[
\frac{\chi(P)}{\rk P}=\frac d2\sum_iw_i\Bigl(\mu(G_i)-\frac\lambda2\Bigr)^2-\frac\sigma8-\frac12\sum_iw_i\bigl(\delta(G_i)-\alpha(G_i)^2\bigr),
\]
which is also Lemma~\ref{lem:rr} at $t=0$, summed over the factors. If all $\mu(G_i)$ lie in $[\frac\lambda2-\frac12,\frac\lambda2+\frac12]$, then $\chi(P)\le\frac{d-\sigma}8\rk P$; for $K_Y\equiv4H$ this is Proposition~\ref{prop:chen-chi}, even for the closed window. In terms of $v$, the condition $\chi(f_*P)\ge0$ reads $v(f_*P)\ge\frac14-(\bar b+\frac32)^2$. Under $K_Y\equiv4H$, for a small CM module, the local input of \cite{Ch} (Lemma~\ref{lem:reflexive}(b), Lemma~\ref{lem:tails} and Proposition~\ref{prop:persist}), if correct, gives $\chi(P)\ge0$ for the normalized extension. It thus replaces the lower bound $v(f_*P)\ge0$ of Lemma~\ref{lem:variance} by $v(f_*P)\ge\frac14-(\bar b+\frac32)^2$, and the bound $v_b\le\frac14$ by $\sum_iw_i(b_i+\frac32)^2\le\frac14$, which needs a window centred at $\frac\lambda2$ (Lemma~\ref{lem:normalize}(b)) rather than an arbitrary window of width one (Lemma~\ref{lem:normalize}(a)). With Proposition~\ref{prop:synthesis} this gives
\[
\frac14\le\frac{d-\sigma}{4d}+\sum_iw_i\Bigl(b_i+\frac32\Bigr)^2\le\frac{d-\sigma}{4d}+\frac14,
\]
i.e.\ $\sigma\le d$. With the mean slope at the centre the two lower bounds differ by exactly $\frac14$; this is the difference between $\sigma\le d$ and $\sigma\le2d$.
\end{remark}

\subsection*{Comparison of the local criteria}

Both local criteria lie inside the graded range $\sigma>d$, but they are incomparable with each other. The quadrangle ring with $n=6$ ($d=1080<\sigma=1872<2160=2d$) is covered locally only by the Main claim of \cite{Ch}. Conversely, the criterion of \cite{OAI} also applies to pairs outside the hypothesis $K\equiv4H$ of \cite{Ch}, such as the quadrangle with $n=7$. The Hesse pairs are covered locally by Theorem~\ref{thm:oai2} for every $n\ge3$; we make this explicit in Corollary~\ref{cor:hesse-local}.

\begin{table}[ht]
\centering
\small
\renewcommand{\arraystretch}{1.15}\setlength{\tabcolsep}{4pt}
\begin{tabular}{@{}lccccc@{}}
	\toprule
	pair $(Y,H)$ & $d=H^2$ & $\sigma$ & $\sigma>d$ & $\sigma>2d$ & $K\equiv4H$, $\sigma>d$ \\
	& & & [An2] (graded) & [OAI] (local) & [Ch] (local, proposed)\\
	\midrule
	Hesse $Y_n$, $4A-E$ & $7n^9$ & $(3n^2-11)n^9$ & $n\ge3$ & $n\ge3$ & never\\
	dual Hesse, $6A-E$ & $24n^6$ & $(2n^2-13)n^6$ & $n\ge5$ & $n\ge6$ & never\\
	quadrangle, $3A-E$ & $5n^3$ & $\frac13(n^2-10)n^3$ & $n\ge6$ & $n\ge7$ & $n=6$\\
	\bottomrule
\end{tabular}
\caption{The three families against the three criteria. On the quadrangle family $K\sim(n-2)H$; on the Hesse and dual Hesse families $K$ is not proportional to $H$. The sources contain the Hesse and quadrangle entries of the graded column \cite[Cors.~1.4, 4.3]{An2}, the quadrangle at $n=6$ in the [Ch] column \cite[Consequence~1.2]{Ch}, the quadrangle at $n=7$ in the [OAI] column, and the non-proportionality of $K$ and $H$ on the Hesse family \cite[Rem.~4.4]{An2}; the other entries apply the criteria here.}\label{tab:coverage}
\end{table}

\begin{corollary}[consequence of {\cite[Thm.~1.2]{OAI}}, proved in Theorem~\ref{thm:local-m}]\label{cor:hesse-local}
For every $n\ge3$, the completed section ring $\Rh(Y_n,4A_n-E_n)$ of the Hesse pair has no small Cohen--Macaulay module. In particular this holds for the ball quotient $Y_3$ of Theorem~\ref{thm:an1}. The same holds for the dual Hesse pairs with $n\ge6$ and for the complete-quadrangle pairs with $n\ge7$.
\end{corollary}

\begin{proof}
By Proposition~\ref{prop:hesse-va}, $H_n$ is very ample, hence ample and globally generated; likewise $6A_n-E_n$ (dual Hesse) and $3A_n-E_n$ (quadrangle) are very ample for all $n\ge2$ \cite[Lemmas~7.1, 7.3]{An1}. By Proposition~\ref{prop:kummer} and Theorem~\ref{thm:an1},
\[
K^2-2c_2=3\sigma=3(3n^2-11)n^9\ \ge\ 48\,n^9\ >\ 42\,n^9=6H_n^2
\]
for $n\ge3$, so Theorem~\ref{thm:oai2} applies. For the dual Hesse pairs, $3(2n^2-13)>6\cdot24$ if and only if $n\ge6$; for the quadrangle, $n^2-10>30$ if and only if $n\ge7$.
\end{proof}

Theorem~\ref{thm:oai2} is the case $m=2$ of Theorem~\ref{thm:local-m}, proved in Section~\ref{sec:highdim}; so the corollary rests on that proof.

Thus the first counterexample to Question~\ref{q:ESW}, the Hesse ball quotient, also yields, through the argument of \cite{OAI}, a counterexample to Hochster's conjecture (Question~\ref{conj:hochster}). The Hesse surfaces are ball quotients only for $n=3$, and $\sigma(Y_n)/H_n^2=(3n^2-11)/7\to\infty$, so the margin over $2d$ grows with $n$.

\section{Higher dimension}\label{sec:highdim}

The local rings of Sections~\ref{sec:chen}--\ref{sec:synthesis} are three-dimensional. We consider two constructions in higher dimension. The first is formal and produces non-isolated singularities (Section~\ref{ssec:powerseries}). The second extends the argument of Section~\ref{sec:oai} to section rings of varieties of dimension $m\ge2$, as \cite{Lo} and \cite[\S5]{An2} extend the projective and graded obstructions, and produces isolated singularities in every dimension. This second route is not contained in the five texts. We give complete proofs; for $m=2$ they supply all details of Theorem~\ref{thm:oai2}.

\subsection{Adjoining power series}\label{ssec:powerseries}
The following is standard; its graded analogue is noted at the end of \cite[\S5]{An2}.

\begin{proposition}\label{prop:powerseries}
Let $(B,\mathfrak n)$ be a Noetherian local ring and $B'=B[[x]]$. If $M$ is a small CM $B'$-module, then $M/xM$ is a small CM $B$-module. Hence, if $B$ has no small CM module, neither has $B[[x_1,\dots,x_k]]$ for any $k\ge0$.
\end{proposition}

\begin{proof}
The ring $B'$ is local with maximal ideal $(\mathfrak n,x)$, $\dim B'=\dim B+1$ and $B'/xB'=B$, so $x$ is part of a system of parameters of $B'$. Every system of parameters of $B'$ is a regular sequence on the maximal Cohen--Macaulay module $M$ \cite[Thm.~2.1.2]{BrHe}, so $x$ is $M$-regular. Hence $M/xM$ is a finitely generated $B$-module, nonzero by Nakayama, with $\depth_B(M/xM)=\depth_{B'}M-1=\dim B$. Induct on $k$.
\end{proof}

\begin{corollary}\label{cor:powerseries}
Let $(\Rh,\widehat\m)$ be a three-dimensional complete normal local domain with residue field $\C$ and without small CM modules, for instance one of the rings of Corollary~\ref{cor:hesse-local}. For every $k\ge1$, $B'=\Rh[[x_1,\dots,x_k]]$ is a complete Noetherian normal local domain of dimension $3+k$, with residue field $\C$, without small CM modules. Its singular locus contains $V(\widehat\m B')$, which has dimension $k$; in particular the singularity is not isolated.
\end{corollary}

\begin{proof}
$B'$ is the completion of the excellent normal local ring $\Rh[x_1,\dots,x_k]_{(\widehat\m,x_1,\dots,x_k)}$, hence a complete normal local domain \cite[\S32]{Mat}. Proposition~\ref{prop:powerseries} excludes small CM modules. Let $\mathfrak p=\widehat\m B'$, a prime with $B'/\mathfrak p=\C[[x_1,\dots,x_k]]$. The local homomorphism $\Rh\to B'_{\mathfrak p}$ is flat, and its closed fibre is the fraction field of $B'/\mathfrak p$. If $B'_{\mathfrak p}$ were regular, $\Rh$ would be regular \cite[Thm.~23.7]{Mat}, and then $\Rh$ itself would be a small CM module. So $B'_{\mathfrak p}$ is singular, and so are all localizations at primes containing $\mathfrak p$.
\end{proof}

\subsection{Setting and statement}
For the rest of this section $X$ is a smooth connected projective variety of dimension $m\ge2$ and $H$ an ample, globally generated divisor on $X$. Put
\[
d=H^m,\qquad \lambda=\frac{K_X\cdot H^{m-1}}{d},\qquad R=R(X,H)=\bigoplus_{t\ge0}H^0\bigl(X,\OO_X(tH)\bigr),\qquad \m=R_+,
\]
and let $\Rh$ be the $\m$-adic completion of $R$. For a torsion-free sheaf $\cG$ of rank $u>0$ on $X$ put
\[
\mu(\cG)=\frac{c_1(\cG)\cdot H^{m-1}}{u\,d},\qquad \nu(\cG)=\frac{c_1(\cG)}{u},\qquad \delta(\cG)=\frac{\Delta(\cG)}{u^2},
\]
where $\Delta(\cG)=2uc_2(\cG)-(u-1)c_1(\cG)^2$.
Harder--Narasimhan filtrations are taken with respect to the slope $\mu$ \cite[\S1.6]{HL}. Recall the excess $I(X,H)=(K_X^2-2c_2(X))\cdot H^{m-2}-(m+1)H^m$ (Lemma~\ref{lem:excess}). By Propositions~\ref{prop:lopez} and~\ref{prop:acm-highdim}, Ulrich bundles (for $H$ very ample) and aCM bundles can exist only if $I(X,H)\le0$.

\begin{theorem}\label{thm:local-m}
If $\Rh(X,H)$ has a small Cohen--Macaulay module, then
\[
\bigl(K_X^2-2c_2(X)\bigr)\cdot H^{m-2}\ \le\ (m+4)\,H^m,\qquad\text{equivalently}\qquad I(X,H)\le3H^m .
\]
\end{theorem}

For $m=2$ one has $K^2-2c_2=3\sigma$, and this is Theorem~\ref{thm:oai2}. The proof follows the argument of \cite{OAI} (Sections~\ref{sec:oai} and~\ref{sec:synthesis}) step by step. Several steps of \cite{OAI} are written there for surfaces but do not depend on the dimension: the model (Proposition~\ref{prop:model-m}, extending \cite[Prop.~2.1]{Ch} and \cite[\S2]{OAI}), the variance lemma (Lemma~\ref{lem:variance-m}, extending \cite[Lemma~4.1]{OAI}) and the modifications (Step~3, extending \cite[Lemma~3.1]{OAI}). The dimension enters through the Riemann--Roch computation of Proposition~\ref{prop:rr-m}, in which the Todd class of $\PP^m$ produces the constant $\frac{m+1}{12}$ in place of $\frac14$; for a trivial push-forward it reduces to the identity in the proof of \cite[Lemma~3.1]{Lo} (Remark~\ref{rem:rr-checks}(iii)).

\subsection{The model}

\begin{proposition}[cf.\ {\cite[Prop.~2.1]{Ch}}, {\cite[\S2]{OAI}} for $m=2$]\label{prop:model-m}
Let $s_0,\dots,s_m\in H^0(X,\OO_X(H))$ be general sections, $f\colon X\to\PP^m$ the morphism they define, $S=\C[x_0,\dots,x_m]\to R$ the homomorphism $x_i\mapsto s_i$, and $A=\C[[x_0,\dots,x_m]]$.
\begin{enumerate}[label=(\alph*)]
	\item $f$ is finite and surjective of degree $d$, and $f^*\OO(1)=\OO_X(H)$.
	\item $R$ is a normal domain, finite over $S$ of generic rank $d$. Its completion is $\Rh=R\otimes_SA$, a complete normal local domain of dimension $m+1$ with residue field $\C$, finite over $A$ of generic rank $d$.
	\item Let $Z=\Bl_0\Spec A$, with exceptional divisor $D\iso\PP^m$, so that $\cI_D|_D\iso\OO_D(1)$. There are a regular integral scheme $W$, a finite morphism $g\colon W\to Z$ and a proper morphism $p\colon W\to\Spec\Rh$ with the following properties. The scheme-theoretic inverse image $g^{-1}(D)$ is isomorphic to $X$, and under this identification $X\subset W$ is an effective Cartier divisor, $g|_X=f$ and $\cI_X|_X\iso\OO_X(H)$. Finally, $p$ induces an isomorphism $W\smallsetminus X\iso\Spec\Rh\smallsetminus\{\widehat\m\}$ lying over the isomorphism $Z\smallsetminus D\iso\Spec A\smallsetminus\{\m_A\}$. In particular $\Rh$ is regular away from the vertex: $\Rh_{\mathfrak p}$ is regular for every prime $\mathfrak p\ne\widehat\m$.
\end{enumerate}
\end{proposition}

\begin{proof}
(a) Since $|H|$ is base-point-free, $m+1$ general sections have no common zero. As $f^*\OO(1)=\OO_X(H)$ is ample, $f$ contracts no curve; a quasi-finite projective morphism is finite, and a finite morphism between varieties of the same dimension is surjective. Its degree is $\int_Xf^*h^m=H^m$, where $h$ is the hyperplane class.

(b) Let $L=\OO_X(H)$, $T=\operatorname{Spec}_X\bigoplus_{t\ge0}L^t$ the total space of $L^{-1}$, $X_0\subset T$ its zero section, and $T^\circ=T\smallsetminus X_0=\operatorname{Spec}_X\bigoplus_{t\in\Z}L^t$. For a homogeneous $s\in R_k$, $k\ge1$, the open set $X_s=X_{s^j}$ is affine, being the complement of a hyperplane section for an embedding by a very ample power $L^{kj}$, and $\Gamma(X_s,L^t)=(R_s)_t$ \cite[Lemma~II.5.14]{Har}; hence the preimage of $D(s)\subset\Spec R$ in $T$ is $\Spec R_s$. These sets cover $T^\circ$ because $L$ is ample, and $X_0$ maps to the vertex, so $T^\circ\iso\Spec R\smallsetminus V(\m)$. Since $H^0(X,\OO_X(tH))=0$ for $t<0$, $R=\Gamma(T^\circ,\OO)$ is the ring of functions of a smooth connected scheme, hence a normal domain. By the projection formula $R=\bigoplus_{t\ge0}H^0(\PP^m,(f_*\OO_X)(t))$, a finitely generated graded $S$-module \cite[Ex.~II.5.9]{Har} of rank $\rk f_*\OO_X=d$. Since $R/(x_0,\dots,x_m)R$ is finite-dimensional, the ideal $(x_0,\dots,x_m)R$ is $\m$-primary, and $\Rh=R\otimes_SA$. As $R_\m$ is excellent and normal, $\Rh$ is normal \cite[\S32]{Mat}, hence a domain, of dimension $m+1$.

(c) Since $\bigoplus_tL^t=f^*\bigoplus_t\OO(t)$, we have $T=X\times_{\PP^m}T_0$ with $T_0=\Tot(\OO_{\PP^m}(-1))=\Bl_0\A^{m+1}$. Put $W=T\times_{\A^{m+1}}\Spec A$ and $Z=T_0\times_{\A^{m+1}}\Spec A=\Bl_0\Spec A$ (blowing up commutes with flat base change). Let $g\colon W\to Z$ be the base change of the finite morphism $T\to T_0$ and $p\colon W\to\Spec(R\otimes_SA)=\Spec\Rh$ that of $T\to\Spec R$. The morphism $T\to\A^{m+1}$ is proper (finite followed by a blow-up), hence so is $W\to\Spec A$; since $\Spec\Rh\to\Spec A$ is finite, hence separated, $p$ is proper as well. Since $\m_A\OO_Z=\cI_D$, the scheme $g^{-1}(D)$ is the fibre of $W$ over the closed point, which equals the fibre of $T$ over $0$, namely $X\times_{\PP^m}D_0=X_0\iso X$, where $D_0\iso\PP^m$ is the zero section of $T_0$. Thus $\cI_X=\cI_D\OO_W$ is locally principal. Its local generators are nonzerodivisors and $W$ is regular along $X$: for a closed point $x\in X\subset W$ with image $x'\in X_0\subset T$, the local homomorphism $\OO_{T,x'}\to\OO_{W,x}$ is flat with closed fibre $\kappa(x')\otimes_SA=\kappa(x')\otimes_\C A/(x_0,\dots,x_m)A=\C$, so $\OO_{W,x}$ is regular of dimension $m+1$ \cite[Thm.~23.7]{Mat} and flatness preserves the nonzerodivisor. Then $\cI_X/\cI_X^2=g^*(\cI_D/\cI_D^2)$, i.e.\ $\cI_X|_X\iso f^*\OO_D(1)=\OO_X(H)$. Since $W$ is proper over $\Spec A$, every closed point of $W$ lies on $X$, so $\OO_{W,x}$ is regular at every closed point $x$. Every point $y$ of the Noetherian scheme $W$ specializes to a closed point $x$, and $\OO_{W,y}$ is a localization of $\OO_{W,x}$, hence regular \cite[Thm.~19.3]{Mat}; so $W$ is regular. Every connected component of $W$ is proper over $\Spec A$, so it contains a closed point and meets the connected $X$; hence $W$ is connected and, being regular, integral. Finally $W\smallsetminus X$ is the base change of $T^\circ\iso\Spec R\smallsetminus V(\m)$ along $\Spec\Rh\to\Spec R$, and the preimage of $V(\m)$ in $\Spec\Rh$ is $\{\widehat\m\}$. The same argument applied to $T_0$ gives $Z\smallsetminus D\iso\Spec A\smallsetminus\{\m_A\}$, compatibly with $g$.
\end{proof}

For $m=2$ this is Proposition~\ref{prop:model} together with the morphisms $f$ and $g$ of Section~\ref{sec:local}, with $X$ here in the role of $Y$ there.

\subsection{The variance lemma on $\PP^m$}
In the rest of this section $h$ denotes the hyperplane class of $\PP^m$. For a coherent sheaf $F$ of rank $N>0$ on $\PP^m$ put
\[
\deg F=c_1(F)\cdot h^{m-1},\qquad v(F)=\frac{2\ch_2(F)\cdot h^{m-2}}{N}-\Bigl(\frac{\deg F}{N}\Bigr)^2 .
\]
For $m=2$ this is the $v$ of Section~\ref{sec:oai}. As there, $v$ is invariant under twists, and $v\bigl(\bigoplus_j\OO(a_j)\bigr)$ is the variance of $(a_j)$.

\begin{lemma}[cf.\ {\cite[Lemma~4.1]{OAI}}]\label{lem:variance-m}
Let $Z$ be an integral Noetherian scheme and $D\subset Z$ an effective Cartier divisor with $D\iso\PP^m$ and $\cI_D|_D\iso\OO_D(1)$. If $\cL'$ is a torsion-free coherent sheaf on $Z$ with $\cL'|_{Z\smallsetminus D}\iso\OO_{Z\smallsetminus D}^N$, then $v(\cL'|_D)\ge0$.
\end{lemma}

\begin{proof}
\emph{Clearing denominators.} Write $\cI=\cI_D$, $\cL=\OO_Z^N$ and $j\colon Z\smallsetminus D\hookrightarrow Z$. A torsion-free sheaf on the integral scheme $Z$ has no nonzero local sections supported on $D$, so $\cL'$ and $\cL$ both embed in $j_*j^*\cL'=j_*j^*\cL$. This sheaf is the increasing union of the coherent subsheaves $\cI^{-q}\cL$, and also of the $\cI^{-q}\cL'$, $q\ge0$: on an affine open set where $D$ is defined by $\tau$, a section of $j_*j^*\cL$ becomes a section of $\cL$ after multiplication by a power of $\tau$. By Noetherianity, $\cI^q\cL\subseteq\cL'\subseteq\cI^{-q}\cL$ for some $q$. Replacing $\cL'$ by $\cI^q\cL'$ replaces $\cL'|_D$ by its twist $\cL'|_D(q)$ and does not change $v$. So we may assume $\cI^c\cL\subseteq\cL'\subseteq\cL$ with $c=2q$.

\emph{The graded pieces.} Let $\cQ=\cL/\cL'$ and $\cQ_j=\cI^j\cQ/\cI^{j+1}\cQ$ for $0\le j<c$. These are coherent $\OO_D$-modules and quotients of $\cI^j\cL/\cI^{j+1}\cL\iso\OO_D(j)^N$. Put $n_j=\rk\cQ_j\le N$. We claim $\deg\cQ_j\ge jn_j$. Indeed $\cQ_j(-j)$ is globally generated, and a globally generated coherent sheaf $G$ on $\PP^m$ has $\deg G\ge0$. To see this, note that its torsion subsheaf contributes an effective divisor class to $c_1$. If $G/G_{\rm tors}$ has rank $n>0$, the surjection from a trivial sheaf induces a nonzero map $\wedge^n\OO^{\oplus M}\to\det(G/G_{\rm tors})$, so the determinant has a nonzero section.

\emph{Derived restriction.} Since $0\to\cI\to\OO_Z\to\OO_D\to0$ is a locally free resolution, $\Tor_i^{\OO_Z}(-,\OO_D)=0$ for $i\ge2$, and $[\cF]_D:=[\cF\otimes\OO_D]-[\Tor_1^{\OO_Z}(\cF,\OO_D)]\in K(D)$ is additive on short exact sequences. For $\cF$ torsion-free a local equation of $D$ is $\cF$-regular, so $\Tor_1=0$. For an $\OO_D$-module $G$ the same resolution gives $\Tor_1(G,\OO_D)\iso G\otimes\cI|_D=G(1)$. Applying this to $0\to\cL'\to\cL\to\cQ\to0$ and to the filtration of $\cQ$,
\[
[\cL'|_D]=[\OO_D^N]+\sum_j\bigl([\cQ_j(1)]-[\cQ_j]\bigr)\qquad\text{in }K(\PP^m).
\]
Since $\ch(\cQ_j(1))-\ch(\cQ_j)=\ch(\cQ_j)\,(e^h-1)=n_jh+\bigl(c_1(\cQ_j)\cdot h+\tfrac{n_j}{2}h^2\bigr)+\dots$, the sheaf $F=\cL'|_D$ has rank $N$, $\deg F=\sum_jn_j$ and
\[
2\ch_2(F)\cdot h^{m-2}=\sum_j\bigl(2\deg\cQ_j+n_j\bigr)\ \ge\ \sum_j(2j+1)\,n_j .
\]
\emph{Conclusion.} Therefore
\[
N^2\,v(F)\ \ge\ N\sum_j(2j+1)n_j-\Bigl(\sum_jn_j\Bigr)^2=\sum_jn_j(N-n_j)+2\sum_{j>\ell}n_j(N-n_\ell)\ \ge\ 0 .\qedhere
\]
\end{proof}

For $m=2$ this is Lemma~\ref{lem:variance}; the proof uses the dimension of $D$ only through the definition of $v$.

\subsection{Riemann--Roch for push-forwards}

\begin{proposition}\label{prop:rr-m}
Let $f\colon X\to\PP^m$ be finite with $f^*\OO(1)=\OO_X(H)$, and let $\cG$ be a torsion-free sheaf of rank $u$ on $X$. Put $\beta=\nu(\cG)-\frac12K_X$ and
\[
\theta(\cG)=\frac{(\beta\cdot H^{m-1})^2-(\beta^2\cdot H^{m-2})\,d}{d}.
\]
Then $f_*\cG$ has rank $ud$, and
\begin{align}
b(\cG):=\frac{\deg f_*\cG}{ud}&=\mu(\cG)-\frac\lambda2-\frac{m+1}{2},\label{eq:rr-slope}\\
v(f_*\cG)&=\frac{m+1}{12}-\frac{(K_X^2-2c_2(X))\cdot H^{m-2}}{12\,d}-\frac{\delta(\cG)\cdot H^{m-2}+\theta(\cG)}{d}.\label{eq:rr-var}
\end{align}
Moreover $\theta(\cG)\ge0$, and if $\cG$ is $\mu$-semistable then $\delta(\cG)\cdot H^{m-2}\ge0$, so that
\[
v(f_*\cG)\ \le\ \frac{m+1}{12}-\frac{(K_X^2-2c_2(X))\cdot H^{m-2}}{12\,d}.
\]
\end{proposition}

\begin{proof}
Let $F=f_*\cG$, of rank $N=ud$. Since $f$ is finite, the projection formula gives $H^i(\PP^m,F(t))=H^i(X,\cG(tH))$, so $\chi(F(t))=\chi(\cG(tH))$ as polynomials in $t$. We compare the coefficients of $t^{m-1}$ and $t^{m-2}$ given by Hirzebruch--Riemann--Roch. The Todd classes are
\begin{gather*}
\operatorname{td}(\PP^m)=\Bigl(\frac{h}{1-e^{-h}}\Bigr)^{m+1}=1+\frac{m+1}2\,h+\frac{(m+1)(3m+2)}{24}\,h^2+\cdots,\\
\operatorname{td}(X)=1-\frac{K_X}2+\frac{K_X^2+c_2(X)}{12}+\cdots,
\end{gather*}
where the $h^2$-coefficient is $\frac{m+1}{12}+\binom{m+1}2\frac14$. In $\chi=\int\ch(\,\cdot\,)\,e^{tH}\operatorname{td}$, the coefficient of $t^k/k!$ is $\int H^k\cdot[\ch(\,\cdot\,)\operatorname{td}]_{m-k}$. With $b=\deg F/N$ and $e=\ch_2(F)\cdot h^{m-2}/N$ we obtain, for $k=m-1$ and $k=m-2$,
\begin{align*}
N\Bigl(b+\frac{m+1}2\Bigr)&=c_1(\cG)\cdot H^{m-1}-\frac u2K_X\cdot H^{m-1},\\
N\Bigl(e+\frac{m+1}2\,b+\frac{(m+1)(3m+2)}{24}\Bigr)&=\Bigl(\ch_2(\cG)-\frac12c_1(\cG)\cdot K_X\Bigr)\cdot H^{m-2}\\
&\qquad+\frac u{12}\bigl(K_X^2+c_2(X)\bigr)\cdot H^{m-2}.
\end{align*}
(For $m=2$ the second line compares constant terms.) The first line is \eqref{eq:rr-slope}. In the second, write $\ch_2(\cG)/u=\frac12(\nu^2-\delta)$, which is the definition of $\Delta$, and divide by $N$:
\begin{multline*}
2e=\frac{(\nu^2-\nu\cdot K_X-\delta)\cdot H^{m-2}}{d}+\frac{(K_X^2+c_2)\cdot H^{m-2}}{6\,d}\\-(m+1)\,b-\frac{(m+1)(3m+2)}{12}.
\end{multline*}
Now $\nu^2-\nu\cdot K_X=\beta^2-\frac14K_X^2$, and by \eqref{eq:rr-slope} $b+\frac{m+1}2=\frac{\beta\cdot H^{m-1}}d$, so $b^2+(m+1)b=\bigl(\frac{\beta\cdot H^{m-1}}d\bigr)^2-\frac{(m+1)^2}4$. Hence
\begin{multline*}
v(F)=2e-b^2=\frac{\beta^2\cdot H^{m-2}}{d}-\frac{(\beta\cdot H^{m-1})^2}{d^2}-\frac{\delta\cdot H^{m-2}}d\\+\Bigl(\frac{K_X^2+c_2}{6}-\frac{K_X^2}{4}\Bigr)\frac{H^{m-2}}{d}+\frac{(m+1)^2}4-\frac{(m+1)(3m+2)}{12}.
\end{multline*}
The first two terms are $-\theta/d$, the bracket is $-\frac1{12}(K_X^2-2c_2)$, and the constants add up to $\frac{m+1}{12}$. This is \eqref{eq:rr-var}.

For $\theta\ge0$: both sides of $(\beta\cdot H^{m-1})^2\ge(\beta^2\cdot H^{m-2})\,H^m$ are homogeneous of degree $2m-2$ in $H$, so we may replace $H$ by a very ample multiple. Let $X_2$ be the intersection of $m-2$ general members of $|H|$. It is smooth by Bertini's theorem \cite[Cor.~III.10.9]{Har} and connected \cite[Cor.~III.7.9]{Har}. Restricting to $X_2$, with $\beta_2=\beta|_{X_2}$ and $H_2=H|_{X_2}$, gives $\beta\cdot H^{m-1}=\beta_2\cdot H_2$, $\beta^2\cdot H^{m-2}=\beta_2^2$ and $d=H_2^2$, and the Hodge index theorem on $X_2$ gives $(\beta_2\cdot H_2)^2\ge \beta_2^2\,H_2^2$. Finally, $\Delta(\cG)\cdot H^{m-2}\ge0$ for $\mu$-semistable torsion-free sheaves is Bogomolov's inequality \cite[Thm.~7.3.1]{HL}.
\end{proof}

\begin{remark}\label{rem:rr-checks}
For $m=2$, $\theta=-\alpha^2$ in the notation of Lemma~\ref{lem:rr}, and \eqref{eq:rr-var} is the equality of Proposition~\ref{prop:synthesis}; its inequality is Lemma~\ref{lem:oai-ss}. Three checks in dimension $m$:
\begin{enumerate}[label=(\roman*)]
	\item for $(X,H)=(\PP^m,\OO(1))$, $f=\mathrm{id}$ and $\cG=\OO$ one has $K^2-2c_2=(m+1)h^2$, and both sides of \eqref{eq:rr-var} vanish;
	\item for $(X,H)=(\PP^m,\OO(2))$, $f=(x_0^2:\dots:x_m^2)$ and $\cG=\OO$, one has $f_*\OO=\bigoplus\OO(-|\varepsilon|/2)$ over the subsets $\varepsilon\subseteq\{0,\dots,m\}$ of even cardinality, whose variance is $\frac{m+1}{16}$ for $m\ge2$. This agrees with \eqref{eq:rr-var}, where $\delta=\theta=0$, $d=2^m$ and $(K^2-2c_2)\cdot H^{m-2}=(m+1)2^{m-2}$. For $m=2$ this is \cite[Rem.~2.7]{An2}.
	\item if $f_*\cG$ is trivial, then $b(\cG)=0$, so $\beta\cdot H^{m-1}=\frac{m+1}2d$ by \eqref{eq:rr-slope}, and $\theta(\cG)=-\alpha'^2\cdot H^{m-2}/u^2$ with $\alpha'=c_1(\cG)-\frac u2\bigl(K_X+(m+1)H\bigr)$; multiplied by $u^2$, \eqref{eq:rr-var} becomes $\Delta(\cG)\cdot H^{m-2}=\alpha'^2\cdot H^{m-2}+\frac{u^2}{12}\bigl((m+1)H^2-K_X^2+2c_2(X)\bigr)\cdot H^{m-2}$, the identity in the proof of \cite[Lemma~3.1]{Lo}.
\end{enumerate}
\end{remark}

\subsection{Proof of Theorem~\ref{thm:local-m}}
Let $M$ be a small CM $\Rh$-module; we use the notation of Proposition~\ref{prop:model-m}.

\emph{Step 1: freeness.} The elements $x_0,\dots,x_m$ form a system of parameters of $\Rh$, hence an $M$-regular sequence. So $\depth_AM=m+1$, and the Auslander--Buchsbaum formula over the regular local ring $A$ gives $M\iso A^N$ for some $N\ge1$.

\emph{Step 2: extension.} Every associated prime $\mathfrak q$ of $M$ satisfies $\dim\Rh/\mathfrak q\ge\depth M=\dim\Rh$ \cite[Prop.~1.2.13]{BrHe}, so $\mathfrak q=0$: $M$ is torsion-free, of some rank $r\ge1$; since $\Rh$ has generic rank $d$ over $A$, $N=rd$. For a prime $\mathfrak p\ne\widehat\m$, $M_{\mathfrak p}$ is Cohen--Macaulay \cite[Thm.~2.1.3]{BrHe} with support $\Spec\Rh_{\mathfrak p}$, hence maximal Cohen--Macaulay over the regular local ring $\Rh_{\mathfrak p}$ (Proposition~\ref{prop:model-m}(c)), hence free. So $\widetilde M$ is locally free of rank $r$ on $U=\Spec\Rh\smallsetminus\{\widehat\m\}\iso W\smallsetminus X$. Let $\cE=(p^*\widetilde M)^{\vee\vee}$. It is torsion-free and $\cE|_{W\smallsetminus X}\iso\widetilde M|_U$. For $x\in W$ one has $\depth\cE_x\ge\min(2,\dim\OO_{W,x})$, since $\cE_x=\operatorname{Hom}(\cE^\vee_x,\OO_{W,x})$ and $\operatorname{Hom}$ into a module of depth $\ge2$ has depth $\ge2$ \cite[Ex.~1.4.19]{BrHe}. If $x\in X$ is not the generic point of $X$, then $\dim\OO_{W,x}\ge2$ and a local equation of $X$ is $\cE_x$-regular, so $\depth(\cE|_X)_x\ge1$. Hence $\cE|_X$ is torsion-free. At the generic point of $X$ the local ring of $W$ is a discrete valuation ring, so $\cE$ is free of rank $r$ there and $\cE|_X$ has rank $r$.

\emph{Step 3: modifications} (\cite[Lemma~3.1]{Ch}, \cite[Lemma~3.1]{OAI}). Let $i\colon X\hookrightarrow W$ be the inclusion, $\cF$ a torsion-free sheaf on $W$ with $\cF|_X$ torsion-free, and $\cF|_X\to\cQ$ a torsion-free quotient with kernel $\cK$. The composite $\cF\to i_*(\cF|_X)\to i_*\cQ$ is surjective; let $\cF'$ be its kernel. It is torsion-free and agrees with $\cF$ off $X$. Tensor $0\to\cF'\to\cF\to i_*\cQ\to0$ with $\OO_X$. A local equation of $X$ is a nonzerodivisor on the torsion-free sheaf $\cF$ on the integral scheme $W$, so $\Tor_1^{\OO_W}(\cF,\OO_X)=0$, while the resolution $0\to\cI_X\to\OO_W\to\OO_X\to0$ gives $\Tor_1^{\OO_W}(i_*\cQ,\OO_X)\iso\cQ\otimes\cI_X|_X=\cQ(H)$. The long exact sequence is $0\to\cQ(H)\to\cF'|_X\to\cF|_X\to\cQ\to0$, that is,
\[
0\to\cQ(H)\to\cF'|_X\to\cK\to0 .
\]
So $\cF'|_X$ is torsion-free of the same rank, and $\mu(\cQ(H))=\mu(\cQ)+1$. Now suppose $\mu_{\max}(P)-\mu_{\min}(P)>1$ for $P=\cF|_X$, and modify along the last Harder--Narasimhan quotient $G$ of $P$, of slope $\mu_{\min}(P)$, with kernel $B$. The new restriction $P'$ is an extension of $B$ by the semistable sheaf $G(H)$ of slope $\mu_{\min}(P)+1<\mu_{\max}(P)$, so $\mu_{\max}(P')\le\mu_{\max}(P)$, while $\mu(P')=\mu(P)+\rk G/\rk P\ge\mu(P)+1/r$. Since $\mu(P')\le\mu_{\max}(P')$ stays bounded by the initial $\mu_{\max}$, the process stops after finitely many steps. Starting from $\cE$ of Step~2, we obtain a torsion-free $\cE$ on $W$ with $\cE|_{W\smallsetminus X}\iso\widetilde M|_U$, whose restriction $P=\cE|_X$ is torsion-free of rank $r$, with Harder--Narasimhan factors $\cG_1,\dots,\cG_s$ satisfying $\mu(\cG_1)-\mu(\cG_s)\le1$. For $m=2$ this is Lemma~\ref{lem:normalize}(a).

\emph{Step 4: $v(f_*P)\ge0$.} The sheaf $\cL'=g_*\cE$ on $Z$ is coherent, and torsion-free because $g$ is finite and dominant and $W$ is integral. Over $Z\smallsetminus D\iso\Spec A\smallsetminus\{\m_A\}$, the morphism $g$ is identified with $\Spec\Rh\smallsetminus\{\widehat\m\}\to\Spec A\smallsetminus\{\m_A\}$ (Proposition~\ref{prop:model-m}(c)), and $\cE|_{W\smallsetminus X}\iso\widetilde M|_U$ is unchanged by Step~3; so $\cL'$ is there the sheaf associated with the $A$-module $M\iso A^N$, hence trivial. Since $g$ is affine and $g^*\OO_D=\OO_X$, $\cL'|_D=g_*(\cE|_X)=f_*P$. Lemma~\ref{lem:variance-m} gives $v(f_*P)\ge0$.

\emph{Step 5: upper bound.} The functor $f_*$ is exact, so $f_*P$ has a filtration with successive quotients $f_*\cG_i$, with weights $w_i=\rk\cG_i/r$. With $b_i=b(\cG_i)$ and $\bar b=\sum_iw_ib_i=\deg f_*P/N$, additivity of rank and $\ch$ gives, as in the proof of Proposition~\ref{prop:synthesis},
\[
v(f_*P)=\sum_iw_i\,v(f_*\cG_i)+\sum_iw_i(b_i-\bar b)^2 .
\]
By \eqref{eq:rr-slope}, $b_i-\mu(\cG_i)$ does not depend on $i$, so the $b_i$ lie in an interval of length $\le1$; with $\gamma$ its midpoint, $\sum_iw_i(b_i-\bar b)^2\le\sum_iw_i(b_i-\gamma)^2\le\frac14$. By Proposition~\ref{prop:rr-m} applied to the semistable $\cG_i$, and by Step~4,
\[
0\ \le\ v(f_*P)\ \le\ \frac{m+1}{12}-\frac{(K_X^2-2c_2(X))\cdot H^{m-2}}{12\,d}+\frac14,
\]
that is, $(K_X^2-2c_2(X))\cdot H^{m-2}\le(m+4)\,d$. \qed

\medskip
If $\Rh(X,H)$ has no small CM module, then neither has $R(X,H)_\m$: by faithful flatness of completion, an $M$-regular sequence of length $m+1$ over $R_\m$ stays regular on $\widehat M$ (as in Corollary~\ref{cor:oai-affine}).

\subsection{One identity in every dimension}\label{ssec:identity}
Let $f\colon X\to\PP^m$ be finite with $f^*\OO(1)=\OO_X(H)$, and let $P$ be a torsion-free sheaf on $X$ with Harder--Narasimhan factors $\cG_i$, weights $w_i$, push-forward slopes $b_i$ with weighted mean $\bar b$, and $v_b=\sum_iw_i(b_i-\bar b)^2$. Proposition~\ref{prop:rr-m} and the decomposition of Step~5 give
\begin{equation}\label{eq:identity-m}
-\frac{I(X,H)}{12}\;=\;d\,\bigl(v(f_*P)-v_b\bigr)\;+\;\sum_iw_i\bigl(\delta(\cG_i)\cdot H^{m-2}+\theta(\cG_i)\bigr),
\end{equation}
and the last sum is $\ge0$. For $m=2$, $-I/12=(d-\sigma)/4$ and \eqref{eq:identity-m} is the identity behind Table~\ref{tab:synthesis}. For the aCM case we need majorization on $\PP^m$.

\begin{lemma}\label{lem:maj-m}
Let $V=\bigoplus_{j=1}^N\OO_{\PP^m}(a_j)$ with $a_1\ge\dots\ge a_N$, and let $0=V_0\subset V_1\subset\dots\subset V_s=V$ be a filtration by subsheaves whose successive quotients have ranks $n_i>0$ and slopes $b_1\ge\dots\ge b_s$ (slope $=\deg/\rk$). Then $\sum_i\frac{n_i}N(b_i+c)^2\le\frac1N\sum_j(a_j+c)^2$ for every $c\in\R$. In particular, with $\bar a=\frac1N\sum_ja_j$, $v_a=\frac1N\sum_j(a_j-\bar a)^2$ and $v_b=\sum_i\frac{n_i}N(b_i-\bar a)^2$, one has $v_b\le v_a$.
\end{lemma}

\begin{proof}
\emph{Step 1.} If $\mathcal S\subset V$ is a subsheaf of rank $p$, then $\deg\mathcal S\le a_1+\dots+a_p$. Indeed $\det\mathcal S=(\wedge^p\mathcal S)^{\vee\vee}$ is invertible, and $\wedge^p\mathcal S\to\wedge^pV=\bigoplus_J\OO(a_J)$, $|J|=p$, is nonzero; so some component $\det\mathcal S\to\OO(a_J)$ is nonzero on the open set where $\mathcal S$ is locally free, whose complement has codimension $\ge2$ since $\mathcal S\subset V$ is torsion-free on the smooth variety $\PP^m$, and extends across it. Hence $\deg\mathcal S\le a_J\le a_1+\dots+a_p$.

\emph{Step 2.} Let $x_1\ge\dots\ge x_N$ be the sequence in which $b_i$ is repeated $n_i$ times, and $\Sigma_x(p)=\sum_{j\le p}x_j$, $\Sigma_a(p)=\sum_{j\le p}a_j$. By Step~1, $\Sigma_x(p)=\deg V_i\le\Sigma_a(p)$ for $p=n_1+\dots+n_i$. Between two such values $\Sigma_x$ is affine and $\Sigma_a$ is concave, so $\Sigma_x\le\Sigma_a$ everywhere, with equality at $p=N$ (Shatz \cite{Sha}).

\emph{Step 3.} Put $S_p=\Sigma_a(p)-\Sigma_x(p)\ge0$, so $S_0=S_N=0$, and $y_p=a_p+x_p+2c$, a non-increasing sequence. Abel summation gives
\[
\sum_j\bigl[(a_j+c)^2-(x_j+c)^2\bigr]=\sum_j(a_j-x_j)\,y_j=\sum_{p<N}S_p\,(y_p-y_{p+1})\ \ge\ 0,
\]
which is Karamata's inequality \cite{Kar}; compare \cite[Lemma~1.5]{La04}. For $m=2$ this is Lemma~\ref{lem:maj}.
\end{proof}

All three obstructions are read off from \eqref{eq:identity-m} in every dimension.
\begin{enumerate}[label=(\roman*),leftmargin=2em]
	\item \emph{Ulrich} ($H$ very ample, $f$ a linear projection). $f_*\cE$ is trivial by Proposition~\ref{prop:ESW}, so $v=v_a=0$, and Lemma~\ref{lem:maj-m} gives $v_b=0$: $\cE$ is $\mu$-semistable (cf.\ \cite[Thm.~2.9]{CH}). Hence $I\le0$, which is the first assertion of Proposition~\ref{prop:lopez}.
	\item \emph{aCM} ($H^i(X,\cE(t))=0$ for $0<i<m$ and all $t$). $f$ is flat by miracle flatness \cite[Thm.~23.1]{Mat}, so $f_*\cE$ is locally free, has no intermediate cohomology and splits by Horrocks' criterion \cite{Hor}, so $v=v_a$. The push-forward slopes of the Harder--Narasimhan factors are decreasing by \eqref{eq:rr-slope}, so Lemma~\ref{lem:maj-m} gives $v_b\le v_a$. Hence $I\le0$, which is Proposition~\ref{prop:acm-highdim}, now without restricting the bundle to a surface section (the Hodge index theorem is still applied on one, in Proposition~\ref{prop:rr-m}). For $m=2$, \eqref{eq:identity-m} is the defect identity of Proposition~\ref{prop:defect}.
	\item \emph{Local.} $v\ge0$ by Lemma~\ref{lem:variance-m} and $v_b\le\frac14$ after modifications. Hence $I\le3d$, which is Theorem~\ref{thm:local-m}.
\end{enumerate}
The gap between the graded and the local bound is always $3H^m$, the same $\frac14$ as for surfaces. Relative to $(m+1)H^m$, however, it shrinks: the ratio of the two bounds on $(K^2-2c_2)\cdot H^{m-2}$ is $\frac{m+4}{m+1}$.

\subsection{Isolated singularities in every dimension}

\begin{corollary}[cf.\ {\cite[Cor.~5.3, Rem.~5.5]{An2}} for the graded part]\label{cor:products}
Let $Y$ be a smooth projective surface, $H_Y$ an ample, globally generated divisor, $d_Y=H_Y^2$ and $\sigma=\sigma(Y)$. For $m\ge2$ let $X=Y\times\PP^{m-2}$ and $H=H_Y\boxtimes\OO(1)$. Then
\[
I(X,H)=3\bigl(\sigma-(m-1)^2d_Y\bigr),\qquad H^m=\binom m2d_Y .
\]
Consequently $(X,H)$ carries no aCM bundle if $\sigma>(m-1)^2d_Y$ (Proposition~\ref{prop:acm-highdim}). If
\[
\sigma\ >\ \frac{(m-1)(3m-2)}{2}\,d_Y,
\]
then the completion $\Rh(X,H)$ of the Segre product $R(Y,H_Y)\#\C[s_0,\dots,s_{m-2}]$ is a complete normal local domain of dimension $m+1$, with residue field $\C$ and an isolated singularity, that has no small CM module.
\end{corollary}

\begin{proof}
Since $T_X=\mathrm{pr}_1^*T_Y\oplus\mathrm{pr}_2^*T_{\PP^{m-2}}$ and $K^2-2c_2=2\ch_2(T)$ is additive,
\[
K_X^2-2c_2(X)=\mathrm{pr}_1^*\bigl(K_Y^2-2c_2(Y)\bigr)+(m-1)\,\mathrm{pr}_2^*h^2,
\]
because $2\ch_2(T_{\PP^k})=\bigl((k+1)^2-(k+1)k\bigr)h^2=(k+1)h^2$. Expand
\[
H^{m-2}=\sum_j\binom{m-2}jH_Y^jh^{m-2-j}.
\]
 Only $j=0$ pairs nontrivially with a $0$-cycle pulled back from $Y$, and only $j=2$ with $h^2$. Hence $(K_X^2-2c_2(X))\cdot H^{m-2}=3\sigma+(m-1)\binom{m-2}2d_Y$ and $H^m=\binom m2d_Y$. Since $(m-1)\binom{m-2}2-(m+1)\binom m2=-3(m-1)^2$, this gives $I(X,H)$. The condition $I>3H^m$ reads $\sigma>\bigl((m-1)^2+\binom m2\bigr)d_Y=\frac{(m-1)(3m-2)}2d_Y$. The divisor $H$ is ample and globally generated, and $R(X,H)$ is the Segre product by the K\"unneth formula. Theorem~\ref{thm:local-m} and Proposition~\ref{prop:model-m}(c) conclude.
\end{proof}

For $m=2$ the two conditions are $\sigma>d_Y$ and $\sigma>2d_Y$. Table~\ref{tab:highdim} applies them to the families of Table~\ref{tab:families}, using that $H_n$ is ample and globally generated (indeed very ample) for all $n\ge2$ (Proposition~\ref{prop:hesse-va} and \cite[Lemmas~7.1, 7.3]{An1}).

\begin{table}[ht]
\centering
\small
\renewcommand{\arraystretch}{1.15}
\begin{tabular}{@{}lccccc@{}}
	\toprule
	& $m=2$ & $m=3$ & $m=4$ & $m=5$ & $m=6$\\
	\midrule
	Hesse, $\sigma/d_Y=(3n^2-11)/7$ & $3\,/\,3$ & $4\,/\,5$ & $5\,/\,7$ & $7\,/\,9$ & $8\,/\,10$\\
	dual Hesse, $\sigma/d_Y=(2n^2-13)/24$ & $5\,/\,6$ & $8\,/\,10$ & $11\,/\,14$ & $15\,/\,18$ & $18\,/\,23$\\
	quadrangle, $\sigma/d_Y=(n^2-10)/15$ & $6\,/\,7$ & $9\,/\,11$ & $13\,/\,16$ & $16\,/\,21$ & $20\,/\,25$\\
	\bottomrule
\end{tabular}
\caption{For $X=Y_n\times\PP^{m-2}$ and $H=H_n\boxtimes\OO(1)$, with $\sigma=\sigma(Y_n)$ and $d_Y=H_n^2$: the smallest exponent $n$ with $\sigma/d_Y>(m-1)^2$, so that $(X,H)$ has no aCM bundle (Proposition~\ref{prop:acm-highdim}) / the smallest $n$ with $\sigma/d_Y>(m-1)(3m-2)/2$, so that $\Rh(X,H)$, of dimension $m+1$, has no small CM module (Corollary~\ref{cor:products}). Both criteria are only sufficient, and since $\sigma/d_Y$ increases with $n$ they then hold for all larger $n$. The column $m=2$ is Table~\ref{tab:coverage}.}\label{tab:highdim}
\end{table}

\begin{corollary}\label{cor:alldim}
For every $m\ge2$ there is a complete Noetherian normal local domain of dimension $m+1$, with residue field $\C$ and an isolated singularity, that has no small Cohen--Macaulay module. For example, take $\Rh(Y_n\times\PP^{m-2},H_n\boxtimes\OO(1))$ for the Hesse surfaces $Y_n$, with $n$ such that $3n^2-11>\frac72(m-1)(3m-2)$.
\end{corollary}

\begin{proof}
Corollary~\ref{cor:products} with $\sigma/d_Y=(3n^2-11)/7$ (Theorem~\ref{thm:an1}).
\end{proof}

The ball quotient $Y_3$ works here only for $m=2$, since $\sigma/d_Y=\frac{16}{7}<7$.

Unlike the rings of Corollary~\ref{cor:powerseries}, these have isolated singularities. For them a small CM module amounts to a nonzero vector bundle $\cV$ on the regular punctured spectrum $U$ with $H^i(U,\cV)=0$ for $1\le i\le m-1$ (take $M=\Gamma(U,\cV)$). This module is finitely generated: it is the module of global sections of the reflexive hull of any coherent extension of $\cV$ to $\Spec\Rh$, since $\Rh$ is normal and $\{\widehat\m\}$ has codimension $\ge2$. This is the analogue of Lemma~\ref{lem:reflexive}, using that maximal Cohen--Macaulay modules over the regular local rings of $U$ are free.

\begin{remark}[Contrast with characteristic $p$]\label{rem:charp}
Let $k$ be an algebraically closed field of characteristic $p>0$, $X$ a smooth connected projective $k$-variety of dimension $m\ge2$ and $H$ an ample Cartier divisor. Then $R(X,H)$ has a graded MCM module, and therefore its completion has a small CM module, by Lemma~\ref{lem:ladder}(b), whose proof (faithful flatness of completion) uses neither the dimension nor the characteristic. For $m=2$ this is the theorem of Hartshorne, Hochster and Peskine--Szpiro \cite[Thm.~1, Cor.~2]{Ho04}; in particular, by Lemma~\ref{lem:ladder}(b), the completion of every three-dimensional graded domain as in \cite[Cor.~2]{Ho04} has a small CM module. For $m\ge3$ it follows from \cite[Cor.~4.5(i)]{ShTa}, which characterizes the existence of a graded MCM module over $R(X,H)$ by the existence of a coherent sheaf on $X$ with MCM stalks, a role played here by $\OO_X$. Alternatively, for every $m\ge2$: since $X$ is Cohen--Macaulay, Serre vanishing and Serre duality show that $B=\{s\in\Z: H^i(X,\OO_X(sH))\ne0$ for some $0<i<m\}$ is finite, say $B\subset[-c,c]$. Let $F$ be the absolute Frobenius of $X$, choose $e$ with $p^e>2c+1$, and put $\cG=F^e_*\OO_X((c+1)H)$, a coherent sheaf since $k$ is perfect. As $F^{e*}\OO_X(H)=\OO_X(p^eH)$, the projection formula gives $\cG(t)=F^e_*\OO_X((c+1+p^et)H)$, and $c+1+p^et\notin[-c,c]$ for every $t\in\Z$. Hence $H^i(X,\cG(t))=0$ for $0<i<m$ and all $t$. Put $q=p^e$. Since $c+1+qt<0$ for $t<0$,
\[
\Gamma_*(\cG)=\bigoplus_{t\ge0}H^0\bigl(X,\OO_X((c+1+qt)H)\bigr),
\]
where $r\in R(X,H)_\ell$ acts by $s\mapsto r^qs$, of degree $t\mapsto t+\ell$ (projection formula). Thus $\Gamma_*(\cG)$ is the part of $F^e_*R(X,H)$ in degrees $\equiv c+1\pmod q$, regraded by $t$, hence finitely generated as $k$ is perfect. It is a nonzero graded MCM $R(X,H)$-module: $H^0_\m$ and $H^1_\m$ vanish as for any module of sections, and $H^{i+1}_\m(\Gamma_*(\cG))\iso\bigoplus_tH^i(X,\cG(t))=0$ for $1\le i\le m-1$.

Over $\C$ both statements fail, even after completion: for $m=2$ by Corollary~\ref{cor:hesse-local}, and for every $m\ge2$ by Corollary~\ref{cor:products}, since $Y_n\times\PP^{m-2}$ is smooth. Thus \cite{An2} shows that in characteristic zero Hochster's conjecture has no graded refinement, and Corollary~\ref{cor:alldim} shows that the local consequence of the graded theorem of characteristic $p$ fails as well, in every dimension. Likewise, the characteristic hypothesis of \cite[Cor.~4.5(i)]{ShTa} is essential not only for its graded statement, as shown in \cite[Cor.~5.4]{An2} with the threefolds $Y_n\times\PP^1$, $n\ge4$, but also for its local consequence, with $Y_n\times\PP^1$, $n\ge5$. The case $n=4$ lies in the gap of item~\ref{oq:highdim} of Section~\ref{sec:questions-open}. As in \cite{An2}, the step that has no analogue in characteristic $p$ is Bogomolov's inequality, used in Proposition~\ref{prop:rr-m}.
\end{remark}
\section{Provenance and status}\label{sec:status}

Figure~\ref{fig:citations} in the introduction records the citations. \cite{Lo} starts from the examples of \cite{An1}; \cite{An2} extends Lopez's inequality to aCM bundles; \cite{Ch} takes its polarized surface from \cite[Rem.~4.2, Cor.~4.3]{An2} and identifies its hypothesis $c>0$ with the graded criterion $d<\sigma$ of \cite{An2} (see \cite[Rem.~1.3]{Ch}), while its own argument does not use the graded theorem. \cite{OAI} attributes the surface inequality without Picard-rank hypothesis to \cite[Prop.~2.1, Rem.~2.3]{An1}, takes the complete-quadrangle construction and polarization from \cite[\S7.2]{An1}, compares with the graded results of \cite{An2}, and acknowledges that its existential claim overlaps with \cite{Ch}.

\subsection*{Status}
To the author's knowledge, none of the five texts has yet appeared in a refereed publication. \cite{An1}, \cite{Lo} and \cite{An2} present their projective and graded statements with complete proofs. \cite{Ch} presents its Main claim as a proposed argument circulated for verification, and states its consequence for Hochster's conjecture conditionally; Section~\ref{sec:chen} states its main steps without claiming independent verification, except that Proposition~\ref{prop:chen-chi} follows from Lemma~\ref{lem:rr} (Remark~\ref{rem:chen-v}). \cite{OAI} claims a complete proof. Its argument is short --- Sections~2--5 of the manuscript, summarized in Sections~\ref{sec:local} and~\ref{sec:oai} above --- and uses only standard tools. Section~\ref{sec:highdim} gives a complete proof of its main theorem, in every dimension. Sections~\ref{sec:synthesis} and~\ref{sec:highdim} represent the new part of this survey.

\section{Open questions}\label{sec:questions-open}

\begin{enumerate}[label=\arabic*.,ref=\arabic*,leftmargin=2em,itemsep=3pt]
	\item\label{oq:sharp} \textbf{The sharp local obstruction.} Does the existence of a small CM module over $\Rh(Y,H)$ force $H^2\ge\sigma(Y)$, as in the graded case? The argument of Lemma~\ref{lem:normalize}(b), a twist followed by modifications, places the normalized Harder--Narasimhan slopes of the restriction $P$ of a reflexive extension (Lemma~\ref{lem:reflexive}(a)) in $[\frac\lambda2-\frac12,\frac\lambda2+\frac12)$, $\lambda=K_Y\cdot H/H^2$, for every $(Y,H)$. By Remark~\ref{rem:chen-v}, it would then suffice that $\chi(P)\ge0$ for such an extension; this is the missing input, and \cite{Ch} claims it when $K\equiv4H$. The quadrangle ring with $n=6$ lies in the gap $d<\sigma<2d$.
	\item \textbf{Beauville's refinement.} Is there a smooth surface $S$ with $\rho(S)=1$ and a very ample $H$ with $H^2<\sigma(S)$?
	\item \textbf{Thresholds.} The Veronese subrings $R(Y_n,H_n)^{(\ell)}$ have no graded MCM module for $7\ell^2<3n^2-11$ and many for $\ell\ge\ell_0$ \cite[Rem.~6.1]{An2}. What happens in between, graded and locally?
	\item\label{oq:bs} \textbf{The Boij--S\"oderberg cone: the second question of Eisenbud and Schreyer.} Eisenbud and Schreyer asked for the Boij--S\"oderberg cone of cohomology tables of aCM bundles on the surfaces of \cite{An1} (see \cite[Acknowledgements]{An2}). For $H^2<\sigma(Y)$ there are no aCM bundles (Theorem~\ref{thm:acm}), so this cone is $\{0\}$, and their question becomes one about all vector bundles. By \cite[Thm.~5.3]{ES11}, the cone of nonnegative rational combinations of cohomology tables of coherent sheaves on $X\subset\PP^N$ coincides with that of $\PP^{\dim X}$ if and only if $X$ carries an Ulrich sheaf; on $\PP^2$ the tables of vector bundles are positive combinations of tables of supernatural bundles \cite{ES09}. For $H^2<\sigma(Y)$, the cone $C_{\mathrm{vb}}(Y,H)$ of nonnegative rational combinations of the tables of vector bundles on $(Y,H)$ contains no nonzero table with vanishing $h^1$-row \cite[Rem.~6.2]{An2}; in particular it misses the supernatural rays of $\PP^2$ with consecutive roots, those of the line bundles $\OO(a)$. What is $C_{\mathrm{vb}}(Y,H)$? For a coherent sheaf $F$ of rank $N>0$ on $\PP^2$ with $b=c_1(F)/N$, Remark~\ref{rem:chen-v} applied to $F(t)$ shows that the roots of $\chi(F(t))$ are $-b-\frac32\pm\frac12\sqrt{1-4v(F)}$; they are real, at distance $\sqrt{1-4v(F)}$, exactly when $v(F)\le\frac14$. For a $\mu_H$-semistable bundle $E$ on $Y$, Proposition~\ref{prop:synthesis} gives $v(f_*E)\le\frac14-\frac\sigma{4d}$, so the Hilbert polynomial of $E$ has real roots at distance at least $\sqrt{\sigma/d}>1$. Is $C_{\mathrm{vb}}(Y,H)$ cut out from the cone of $\PP^2$ by conditions of this kind, once unstable bundles (the term $v_b$) are taken into account? For $(Y,\ell H)$ the bound $\sqrt{\sigma/(\ell^2d)}$ drops below $1$ exactly when $\ell^2H^2>\sigma$, consistent with the existence of Ulrich bundles for $\ell\gg0$ \cite{CoHu}; how does the cone fill up as $\ell$ grows?
	\item\label{oq:highdim} \textbf{Higher dimension.} Theorem~\ref{thm:local-m} gives $I(X,H)\le3H^m$ for small CM modules over $\Rh(X,H)$, while aCM bundles require $I(X,H)\le0$. For the products of Corollary~\ref{cor:products} this leaves the range $(m-1)^2<\sigma(Y)/H_Y^2\le(m-1)(3m-2)/2$ open, e.g.\ the Hesse pairs with $(m,n)=(3,4)$. Is $I(X,H)\le0$ necessary locally, as item~\ref{oq:sharp} asks for surfaces? Does Chen's argument extend beyond surfaces?
	\item \textbf{Mixed characteristic.} After spreading out, the reductions modulo primes of good reduction of the rings above have graded MCM modules, hence their completions have small CM modules (Remark~\ref{rem:charp}); Bogomolov's inequality is the step that fails. Are there analogous counterexamples in mixed characteristic, where the conjecture was originally motivated by Serre's positivity of intersection multiplicities?
	\item \textbf{A characterization.} For which smooth projective surfaces $Y$ and ample, globally generated $H$ does $\Rh(Y,H)$ admit a small CM module? By Theorem~\ref{thm:oai2}, $\sigma\le2H^2$ is necessary, and by Coskun--Huizenga existence holds for $(Y,\ell H)$ with $\ell\gg0$.
	\item \textbf{Non-graded examples.} Are there complete local domains without small CM modules that are not completions of graded rings, in particular not of the form $\Rh(X,H)$?
\end{enumerate}

\section*{Statements and declarations}
\noindent\textbf{Use of generative AI.} ChatGPT and Claude were used for exploratory computations, literature search and editorial assistance, including drafting of the proofs in Sections~\ref{sec:synthesis} and~\ref{sec:highdim}. The author is solely responsible for the final text.

\end{document}